\documentclass[11pt]{article}
\usepackage[utf8]{inputenc} 
\usepackage{latexsym,amsmath,amssymb,textcomp}
\usepackage[all]{xy}
\usepackage[nottoc, notlof, notlot]{tocbibind}
\usepackage{geometry} 
\usepackage{graphicx} 
\usepackage{mathtools}
\usepackage{amsthm}
\usepackage{amsmath}
\usepackage{booktabs}
\usepackage{array} 
\usepackage{paralist}
\usepackage{verbatim} 
\usepackage{subfig} 
\usepackage{ stmaryrd }
\usepackage{hyperref}
\usepackage{fancyhdr} 
\title{Equivalences between blocks of cohomological Mackey algebras.}
\author{Baptiste Rognerud}
\begin{document}
\maketitle
\theoremstyle{plain}
\newtheorem{theo}{Theorem}[section]
\newtheorem{theox}{Theorem}[section]
\renewcommand{\thetheox}{\Alph{theox}}
\newtheorem{prop}[theo]{Proposition}
\newtheorem{lemma}[theo]{Lemma}
\newtheorem{coro}[theo]{Corollary}
\newtheorem{question}[theo]{Question}
\newtheorem{notations}[theo]{Notations}
\theoremstyle{definition}
\newtheorem{de}[theo]{Definition}
\newtheorem{dex}[theox]{Definition}
\theoremstyle{remark}
\newtheorem{ex}[theo]{Example}
\newtheorem{re}[theo]{Remark}
\newtheorem{res}[theo]{Remarks}
\renewcommand{\labelitemi}{$\bullet$}
\newcommand{\comu}{co\mu}
\newcommand{\mo}{\ensuremath{\hbox{mod}}}
\newcommand{\Mo}{\ensuremath{\hbox{Mod}}}
\newcommand{\decale}[1]{\raisebox{-2ex}{$#1$}}
\newcommand{\decaleb}[2]{\raisebox{#1}{$#2$}}
\begin{abstract}
Let $G$ be a finite group and $(K,\mathcal{O},k)$ be a $p$-modular system ``large enough". Let $R=\mathcal{O}$ or $k$. There is a bijection between the blocks of the group algebra $RG$ and the central primitive idempotents (the blocks) of the so-called cohomological Mackey algebra $co\mu_{R}(G)$. Here, we prove that a so-called \emph{permeable} derived equivalence between two blocks of group algebras implies the existence of a derived equivalence between the corresponding blocks of cohomological Mackey algebras. In particular, in the context of Brou\'e's abelian defect group conjecture, if two blocks are \emph{splendidly} derived equivalent, then the corresponding blocks of cohomological Mackey algebras are derived equivalent. 
\end{abstract}
\par\noindent
{\it{\footnotesize   Key words: Modular representation. Finite group. Mackey functor. Block theory.}}
\par\noindent
{\it {\footnotesize A.M.S. subject classification: 20C05, 18E30,16G10.}}

\section{Introduction.}
The notion of Mackey functor, introduced by Green in \cite{green}, is a generalization of linear representations of a finite group $G$. A Mackey functor, for Green, is the data of a representation of $N_{G}(H)$ for every subgroup $H$ of $G$, together with relations between these representations. A couple of years later, Dress gave a completely different, but equivalent, definition using the formalism of categories. Twenty years later, Th\'evenaz and Webb introduced the Mackey algebra and proved that a Mackey functor is nothing but a module over this algebra. 
\newline A Mackey functor is \emph{cohomological} if its restriction and induction maps behave like those of the cohomology of groups. The category consisting of cohomological Mackey functors is a full subcategory of the category of Mackey functors. This category is equivalent to the category of modules over the so-called cohomological Mackey algebra. 
Let $R$ be a commutative ring. The cohomological Mackey algebras share a lot of properties of  group algebras, for example $co\mu_{R}(G)$ is $R$-free of finite rank and this rank is independent of the ring $R$. Moreover if $R$ is a field of characteristic which does not divide the order of $G$, then $co\mu_{R}(G)$ is semi-simple. When $(K,\mathcal{O},k)$ is a $p$-modular system, it is possible to define a decomposition theory for $co\mu_{\mathcal{O}}(G)$, in particular the Cartan matrix of the cohomological Mackey algebra is symmetric. However there are some differences with group algebras: most of the time the determinant of the Cartan Matrix of $co\mu_{k}(G)$ is zero. Moreover the cohomological Mackey algebra is not a symmetric algebra. 
\newline It has been noticed for a long time that there are deep links between the representation's theory of finite groups and the theory of Mackey functors. Some objects of the first theory are much more natural when you see them via the Mackey functors' theory (e.g. $p$-permutation modules, Brauer quotient, $\cdots$). It is quite natural to think that this theory may be used in order to understand some open questions of the representation's theory of finite group. The first  attempts was about Alperin weight's conjecture. Jacques Th\'evenaz and Peter Webb proved that this conjecture is equivalent to a conjecture on Mackey functors. Here the arithmetic of the first conjecture is encoded in two Mackey functors. 
\newline In this paper, we propose to look at Brou\'e's abelian defect group conjecture and try to see if the effect of the equivalence conjectured by Brou\'e on the Mackey algebras. 
\newline Let $R=\mathcal{O}$ or $k$. In their paper, Th\'evenaz and Webb proved that there is a bijection between the blocks of $RG$ and the primitive central idempotents of the so-called $p$-local Mackey algebra $\mu_{R}^{1}(G)$. In the proof, they remark that there is also a bijection between the blocks of $RG$ and the blocks of the cohomological Mackey algebra $co\mu_{R}(G)$. Let us denote by $b\mapsto \iota(b)$ this bijection. 
\newline Using the Brauer correspondence, we have the following diagram: Let $b$ be a block of $RG$ with defect group $D$ and $b'$ be its Brauer correspondent in $RN_{G}(D)$. 
\begin{equation*}
\xymatrix{
b\in Z(RG)\ar[r]\ar[d] & \iota(b)\in Z(co\mu_{R}(G)\ar[d]) \\
b'\in Z(RN_{G}(D))\ar[r] & \iota(b')\in Z(co\mu_{R}(N_{G}(D))).
}
\end{equation*}
If $D$ is abelian, it is conjectured by Brou\'e that the block algebras $RGb$ and $RN_{G}(D)b'$ are deeply connected.  It is a very natural question to ask if the same can happen for the corresponding Mackey algebras. However, we should notice that, since the cohomological Mackey algebra is not symmetric, the usual stable category is not triangulated, so we decided not to look at stable equivalences.
\newline In this article we will focus on Morita equivalences and derived equivalences. 
\begin{question}[Bouc]
Let $G$ be a finite group and $b$ be a block of $\mathcal{O}G$ with abelian defect group $D$. Let $b'$ be the Brauer correspondant of $b$ in $\mathcal{O}N_{G}(D)$. Is there a derived equivalence $D^{b}(co\mu_{\mathcal{O}}(G)\iota(b))\cong D^{b}(co\mu_{\mathcal{O}}(N_{G}(D))\iota(b'))$ ?
\end{question}
The main result of this paper is the following theorem which settles the question for the cohomological Mackey algebra in the case of a \emph{splendid} equivalence (see \cite{splendid}): 
\begin{theo}\label{thea}
Let $G$ and $H$ be two finite groups, let $b$ be a block of $RG$ and $c$ be a block of $RH$. If $RGb$ and $RHc$ are splendidly derived equivalent, then $$D^{b}(co\mu_{R}(G)\iota(b))\cong D^{b}(co\mu_{R}(H)\iota(c)).$$ 
\end{theo}
In the fist part, we recall Yoshida's point of view on cohomological Mackey functors. There are several points of view on the notion of Mackey functors, so there are several points of view on the notion of cohomological Mackey functors. There are technical issues about the different versions of Yoshida's equivalence. A systematic use of the Burnside functor will clarify the situation. 
In the second part, we give an explicit isomorphism between the center of the group algebra and the center of the cohomological algebra. With this isomorphism we have a description of the blocks of the cohomological Mackey algebras. This decomposition is compatible with the block decomposition of the category of cohomological Mackey functors introduced by Th\'evenaz and Webb. With this description, we prove a block version of Yoshida's theorem.  
 \newline\indent Using this block version of the \emph{Yoshida} equivalence, we see that a so-called \emph{permeable} Morita (resp. derived) equivalence between blocks of group algebras can be lifted to a Morita (resp. derived) equivalence between the corresponding blocks of cohomological Mackey algebras. For example splendid Morita equivalences, and splendid derived equivalences can be lifted. Even if the notion of permeable equivalence is very natural it seems to the author that it has not been considered yet. We investigate on the very basic properties of these equivalences. In particular, we show that in general, Morita equivalences are not permeable and we give an example of permeable Morita equivalence which is not splendid. 
 \newline\indent We give two applications of Theorem \ref{thea}. The first one is a new point of view on Bouc's Theorem about the determinant of the Cartan matrices of the blocks of the cohomological Mackey algebras. He proved that this determinant is non zero if and only if the block is nilpotent with a cyclic defect group. The proof is based on a combinatorial approach and it may be surprising that the nilpotent blocks appear here. We show that it is in fact very natural and comes from a structural reason. Finally we give an extremely naive application of Theorem \ref{thea} to representation of finite groups. If the Cartan matrices of two blocks $co\mu_{R}(G)\iota(b)$ and $co\mu_{R}(H)\iota(c)$ are not the same, then $RGb$ and $RHc$ are not splendidly (or permeable) Morita equivalent. This is a sufficient criterion for two blocks to not be splendidly Morita equivalent. This is particularly useful since it is possible to compute these matrices via an algorithm (in GAP4 e.g.). We give a particularly surprising example of nilpotent blocks with quaternion defect group, which was discover by using this method.  
 \begin{re}
 The purpose of two first parts of this paper is to investigate on the blocks of the cohomological Mackey algebra and to prove a block version of Yoshida's equivalence. If this proof involves rather technical discussion about Mackey functors, the result is not technical at all. Here, we do not assume the reader familiar with any deep result on Mackey functors. Still, if the reader is more interest by the link between splendid equivalences and equivalences between blocks of cohomological Mackey algebras, he might take Corollary \ref{yoshida_co} as a definition. 
 \end{re}
 \paragraph{Notations:} Let $R$ be a commutative ring with unit. We denote by $R$-$Mod$ the category of (all) $R$-modules and by $R$-$mod$ the category consisting of the finitely generated $R$-modules. We denote by $proj(R)$ the category of finitely generated projective $R$-modules.
\newline Let $G$ be a finite group then a \emph{permutation projective} $RG$-module is a direct summand of a permutation module. Let $p$ be a prime number. We denote by $(K,\mathcal{O},k)$ a $p$-modular system, i-e $\mathcal{O}$ is a complete discrete valuation ring with maximal ideal $\mathfrak{p}$, such that $\mathcal{O}/\mathfrak{p}=k$ is a field of characteristic $p$ and $Frac(\mathcal{O})=K$ is a field of characteristic zero. If $R=\mathcal{O}$ or $k$, then the permutation projective $RG$-modules are called $p$-\emph{permutation} modules. We denote by $G$-$set$ the category of finite $G$-sets. If $H$ is a subgroup of $G$ then, we denote by $N_{G}(H)$ its normalizer in $G$. The quotient $N_{G}(H)/H$ will be, sometimes, denoted by $\overline{N}_{G}(H)$. If $G$ is a finite group, the union of all transitive $G$-sets is denoted by $\Omega_{G}$. 
\newline If $\mathcal{A}$ is an abelian category, we denote by $C^{-}(\mathcal{A})$ the category of right bounded complexes of $\mathcal{A}$, and by $C^{b}(\mathcal{A})$ the category of right and left bounded complexes of $\mathcal{A}$. We denote by $K^{-}(\mathcal{A})$ and $K^{b}(\mathcal{A})$ the corresponding homotopy categories, and finally by $D^{-}(A)$ and $D^{b}(A)$ the corresponding derived categories. Moreover, if $A$ is an $R$-algebra, we denote by $D^{s}(A)$ the derived category $D^{s}(A$-$Mod)$ for $s= b$ or $s=-$. Finally, if $X$ is an $A$-module (resp. a bounded complex of $A$-modules), we denote by $X^{*}$ the $R$-linear dual of $X$.
\newline If $F : \mathcal{A}\to \mathcal{B}$ and $G: \mathcal{B}\to \mathcal{A}$ are two functors, we denote by $F \dashv G$ the fact that $F$ is a left adjoint of $G$.  
\newline N.B. We will denote by the same letter the block idempotents for the ring $\mathcal{O}$ and the field $k$.  

\section{Yoshida's point of view on cohomological Mackey functors.}
\subsection{Basic definitions.}
For basic definitions of Mackey functors, we refer the reader to Section $2$ of \cite{tw}. In this paper we will use Dress' point of view and Th\'evenaz-Webb's point of view. We will use Green's point of view only for the definition of cohomological Mackey functors since it is much more natural. Here, we just recall the definition of the Mackey algebra.
Let $R$ be a commutative ring with unit. 
\begin{de}
The Mackey  algebra $\mu_{R}(G)$ for $G$ over $R$ is the unital associative algebra with generators  $t_{H}^{K}$, $r^{K}_{H}$ and $c_{g,H}$ for $H\leqslant K\leqslant G$ and $g\in G$, with the following relations:
\begin{itemize}
\item $\sum_{H\leqslant G}t^{H}_{H}=1_{\mu_{R}(G)}$. 
\item $t^{H}_{H}=r^{H}_{H}=c_{h,H}$ for $H\leqslant G$ and $h\in H$. 
\item $t^{L}_{K}t_{H}^{K}=t^{L}_{H}$, $r^{K}_{H}r^{L}_{K}=r^{L}_{H}$ for $H\subseteq K\subseteq L$. 
\item $c_{g',{^{g}H}}c_{g,H}=c_{g'g,H}$, for $H\leqslant G$ and $g,g'\in G$. 
\item $t^{{^{g}K}}_{{^{g}H}}c_{g,H}=c_{g,K}t^{K}_{H}$ and $r^{{^{g}K}}_{{^{g}H}}c_{g,K}=c_{g,H}r^{K}_{H}$, $H\leqslant K$, $g\in G$. 
\item $r^{H}_{L}t^{H}_{K}=\sum_{h\in [L\backslash H / K]} t^{L}_{L\cap {^{h} K}} c_{h, L^{h} \cap K} r^{K}_{L^{h}\cap K}$ for $L\leqslant H \geqslant K$. 
\item All the other products of generators are zero. 
\end{itemize}
\end{de}
\begin{de}
A Mackey functor for $G$ over $R$ is a left $\mu_{R}(G)$-module.
\end{de}
\begin{prop}\label{basis}
The Mackey algebra is a free $R$-module, of finite rank independent of $R$. The set of elements $t^{H}_{K}xr^{L}_{K^{x}}$, where $H$ and $L$ are subgroups of $G$, where $x\in [H\backslash G/L]$, and $K$ is a subgroup of $H{\cap~{\ ^{x}L}}$ up to $(H\cap {\ ^{x}L})$-conjugacy, is an $R$-basis of $\mu_{R}(G)$.  
\end{prop}
\begin{proof}
Section $3$ of \cite{tw}. 
\end{proof}
Now, let us recall the definition of the Burnside group of a finite $G$-set. 
\begin{de}[2.4.1 \cite{bouc_green}]\label{burnside}
If $X$ is a finite $G$-set, the category of $G$-sets over $X$ is the category with objects $(Y,\phi)$ where $Y$ is a finite $G$-set and $\phi$ is a morphism from $Y$ to $X$. A morphism $f$ from $(Y,\phi)$ to $(Z,\psi)$ is a morphism of $G$-sets $f:Y\to Z$ such that $\psi\circ f=\phi$. 
\newline\indent The Burnside group of $X$, denoted by $B(X)$, is the Grothendieck group of the category of $G$-sets over $X$, for relations given by disjoint union. Moreover, we denote by $RB(X)$ the Burnside group after scalars extension. That is $RB(X) = R\otimes_{\mathbb{Z}}B(X)$.
\end{de}
\begin{re}
 If $X$ is a $G$-set, the Burnside group $RB(X^2)$ has a ring structure. A $G$-set $Z$ over $X\times X$ is the data of a $G$-set $Z$ and a map $(b\times a)$ from $Z$ to $X\times X$, denoted by $(X\overset{b}{\leftarrow} Y \overset{a}{\rightarrow} X)$. The product of (the isomorphism class of ) $(X\overset{\alpha}{\leftarrow} Y \overset{\beta}{\rightarrow} X)$ and (the isomorphism class of )$(X\overset{\gamma}{\leftarrow} Z \overset{\delta}{\rightarrow} X)$ is given by (the isomorphism class of) the pullback along $\beta$ and $\gamma$. 
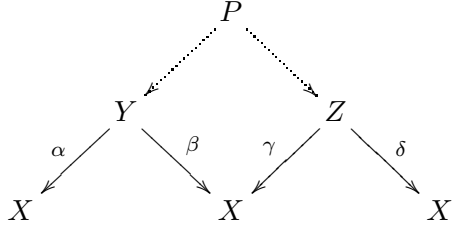
\begin{equation*}
\xymatrix{
& & P\ar@{..>}[dr]\ar@{..>}[dl] & & \\
& Y\ar[dl]_{\alpha}\ar[dr]^{\beta} & & Z\ar[dl]_{\gamma}\ar[dr]^{\delta} & \\
X & & X & &X
}
\end{equation*}
The identity of this ring is (the isomorphism class) $
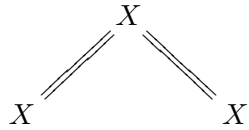
\xymatrix{&X\ar@{=}[rd]\ar@{=}[dl]&\\X& &X }$
\newline In the rest of the paper, we will denote by the same symbol a $G$-set over $X\times X$ and its isomorphism class in $RB(X\times X)$.
\end{re}
Let us recall that the Mackey algebra is isomorphic to a Burnside algebra:
\begin{prop}[Proposition 4.5.1 \cite{bouc_green}]\label{burnside_alg}
The Mackey algebra $\mu_{R}(G)$ is isomorphic to $RB(\Omega_{G}^{2})$, where $\Omega_{G}=\sqcup_{L\leqslant G} G/L$. 
\end{prop}
\begin{proof}
Let $H\leqslant K$ be two subgroups of $G$, then we denote by $\pi^{K}_{H}$ the natural surjection from $G/H$ to $G/K$. If $g\in G$, then we denote by $\gamma_{H,g}$ the map from $G/\ ^{g}H$ to $G/H$ defined by $\gamma_{H,g}(xgHg^{-1}) =xgH$. The isomorphism $\beta$ is defined on the generators of $\mu_{R}(G)$ by: $$\beta(t_{H}^{K}) = \xymatrix{ & G/H\ar[dl]_{\pi^{K}_{H}}\ar@{=}[rd]  & \\ \Omega_{G}\supset G/K&& G/H\subset \Omega_{G}}$$
$$\beta(r_{H}^{K}) = \xymatrix{ & G/H\ar[dr]^{\pi^{K}_{H}}\ar@{=}[ld]  & \\ \Omega_{G}\supset G/H&& G/K\subset \Omega_{G}}$$
$$\beta(c_{g,H}) =  \xymatrix{ & G/\ ^gH\ar[dr]^{\gamma_{H,g}}\ar@{=}[ld]  & \\ \Omega_{G}\supset G/\ ^gH&& G/H\subset \Omega_{G}}$$
\end{proof}
For basic results about cohomological Mackey functors see Section $16$ of \cite{tw}. 
\newline A Mackey functor, in the sense of Green, is cohomological if whenever $K\leqslant H\leqslant G$, one has $t_{K}^{H}r^{H}_{K} = |H:K|Id_{M(H)}$. Let us denote by $Comack_{R}(G)$ the full subcategory consisting of cohomological Mackey functors. The category $Comack_{R}(G)$ is equivalent to the category of modules over the so-called cohomological Mackey algebra, denoted by $co\mu_{R}(G)$. The cohomological Mackey algebra is the quotient of the Mackey algebra $\mu_{R}(G)$ by the ideal generated by the $t_{K}^{H}r^{H}_{K} - |H:K|t^{H}_{H}$ for $K\leqslant H\leqslant G$. If $x\in \mu_{R}(G)$, we denote by $\overline{x}$ its image in the quotient $co\mu_{R}(G)$. 
\subsection{Yoshida's equivalence.}
In this section we recall Yoshida's theorem for cohomological Mackey functors. This theorem says that the category of cohomological Mackey functors for a group $G$ over a ring $R$, in the sense of Dress, is equivalent to the category of $R$-linear contravariant functors from the category of permutation projective modules to the category of $R$-modules. There are several points of view on the notion of Mackey functors, so for each of these points of view, we have a version of the Yoshida's theorem. In general it is not easy to \emph{move} between these several versions. Since we will use in the next section an explicit version of Yoshida's theorem for the modules over the cohomological algebra \emph{and} for Dress' point of view, we recall here how the Yoshida's equivalence is defined. We believe that a systematic use of the Burnside functor will clarify the link between these different versions of Yoshida's Theorem.
\newline The main tool is the so-called linearization Mackey functor:
\begin{lemma}
Let $X$ be a finite $G$-set. We set $\Pi(X)=RX$, that is the permutation $RG$-module with basis $X$. Let $f: X\to Y$ be a morphism of $G$-sets. Then we have a morphism of $RG$-modules $\Pi^{*}(f) : RY\to RX$ defined as follows:
\begin{equation*}
\Pi^{*}(f)\big(\sum_{y\in Y} r_{y}y\big)=\sum_{x\in X}r_{f(x)}x.
\end{equation*}
On the other direction, we have a morphism $\Pi_{*}(f) : RX\to RY$ defined as follows:
\begin{equation*}
\Pi_{*}(f)(\sum_{x\in X} r_{x}x) = \sum_{x\in X}r_{x}f(x). 
\end{equation*}
The bivariant functor $\Pi$ is a (non commutative) Mackey functors with values in the category $RG$-Mod, i-e we have:
\begin{itemize}
\item The bivariant functor $\Pi$ is additive.
\item If 
\begin{equation*}
\xymatrix{
X\ar[r]^{a}\ar[d]^{b} & Y\ar[d]^{c}\\
Z\ar[r]^{d}& T
}
\end{equation*}
is a pullback diagram of $G$-sets, then 
\begin{equation*}
\xymatrix{
RX\ar[d]^{\Pi_{*}(b)} & RY\ar[d]^{\Pi_{*}(c)}\ar[l]_{\Pi^{*}(a)}\\
RZ& RT\ar[l]^{\Pi^{*}(d)}
}
\end{equation*}
is a commutative diagram.
\end{itemize}
\end{lemma}
\begin{proof}
Clear. 
\end{proof}
If the context is clear, we will simply denote by $f^{*}$ the morphism $\Pi^{*}(f)$ and by $f_{*}$ the morphism $\Pi_{*}(f)$. 
\begin{de}
Let $G$ be a finite group and $R$ be a commutative ring with unit. Then $End_{RG}(R\Omega_{G})$ is the Yoshida algebra for the group $G$ over the ring $R$. The product is defined by $f \times g = g\circ f$, for $f,g \in End_{RG}(R\Omega_{G})$.
\end{de}
\begin{lemma}\label{lin_proj}
Let $X$ and $Y$ be to finite $G$-sets. Then there is a surjective map $p_{L}$, called the linear projection from $RB(X\times Y)$ to $Hom_{RG}(RX,RY)$, defined on a $G$-set over $X\times Y$ by:
\begin{equation*}
p_{L}( X\overset{b}{\leftarrow} Z \overset{a}{\rightarrow} Y) = a_{*}\circ b^{*}: RX\to RZ\to RY.
\end{equation*}
\end{lemma}
\begin{proof}
By additivity, it is enough to check the result for two transitive $G$-sets. Let $H$ and $K$ be two subgroups of $G$. Let us suppose that $X=G/H$ and $Y=G/K$. Let $Z_{H,K,x}$ be the following $G$-set over $G/H\times G/K$:
\begin{equation}\label{ba}
\xymatrix{
& G/H\cap\ ^{x}K\ar[ld]_{\pi^{H}_{H\cap{\ ^{x}K}}} \ar[dr]^{\ \ \pi^{K}_{K\cap H^{x}} \circ \gamma_{H^{x}\cap K,x}} & \\
G/H & & G/K
}
\end{equation}
where the maps denoted by $\pi$ are the natural projections and the map denoted by $\gamma_{H^x\cap K,x}$ is defined by $$\gamma_{H^{x}\cap K,x}(gH\cap\ ^xK)= gxH^{x}\cap K.$$ 
Then one can check that: $$p_{L}(Z_{H,K,x})(gH)=\sum_{h\in [H/H\cap K^{x}]} ghxK.$$ Moreover, the isomorphism class of this $G$-set over $G/H\times G/K$ depends only on the double coset $HxK$. We will still denote by $Z_{H,K,x}$ the image of this $G$-set in the Burnside group $RB(G/H\times G/K)$. 
\newline The result now follows from Lemma $3.1$ of \cite{yoshida_g_functors_II}, which says that the set of morphisms $p_{L}(Z_{H,K,x})$ when $x$ runs through a set of representatives of the double cosets $H\backslash G/K$ is a $R$-basis of $Hom_{RG}(RG/H,RG/K)$. 
\end{proof}
This linear projection is compatible with the composition of the morphisms in the following sense:
\begin{lemma}
Let $U_{a,b} = (X\overset{b}{\leftarrow} U \overset{a}{\rightarrow} Y)$ be a $G$-set over $X\times Y$. Let $V_{c,d} = (Y\overset{d}{\leftarrow} V \overset{c}{\rightarrow} Z)$ be a $G$-set over $Y\times Z$. Then
\begin{equation*}
p_{L}(X\overset{b}{\leftarrow} U \overset{a}{\rightarrow} Y)\times p_{L}(Y\overset{d}{\leftarrow} V \overset{c}{\rightarrow} Z) = p_{L}(U_{a,b}\times V_{c,d}),
\end{equation*}
where the product $U_{a,b} \times V_{c,d}$ is as in Definition \ref{burnside}, that is the pullback along the morphisms $a$ and $d$.
\end{lemma}
\begin{proof}
This follows from the pullback property of the bivariant functor $\Pi = (\Pi^{*},\Pi_{*})$.
\end{proof}
\begin{theo}[Yoshida's Theorem for cohomological Mackey algebra]\label{yo1}\label{yoshida_alg}
Let $G$ be a finite group and $R$ be a commutative ring with unit. Then, there is an isomorphism of algebras $\phi : co\mu_{R}(G)\to End_{RG}(R\Omega_{G})$, which makes the following diagram commutative:
\begin{equation}\label{dia}
\xymatrix{ 
\mu_{R}(G)\ar[d]^{p}\ar[rr]^{\beta} && RB(\Omega_{G}\times \Omega_{G})\ar[d]^{p_{L}}\\
co\mu_{R}(G) \ar[rr]^{\phi} && End_{RG}(R\Omega_{G}).
}
\end{equation}
Here, the map $p:\mu_{R}(G)\to co\mu_{R}(G)$ is the natural projection. The map $$\beta: \mu_{R}(G) \to RB(\Omega_{G}\times \Omega_{G}),$$ is the isomorphism of Proposition \ref{burnside_alg}, and $p_{L}$ is the map of Lemma \ref{lin_proj}.
\end{theo}
\begin{proof}
The isomorphism $\phi : co\mu_{R}(G)\to End_{RG}(R\Omega_{G})$ is defined as follows:\\ let $x\in \mu_{R}(G)$. Then $$\phi(p(x)):= p_{L}(\beta(x)).$$
\begin{itemize}
\item The morphism $\phi$ is well defined since $p_{L}\big(\beta(t^{K}_{H}r^{K}_{H})\big) = |K:H|p_{L}\big(\beta(t_{K}^{K})\big)$.
\item Since $p_{L}$ and $\beta$ are two morphisms of algebras, the map $\phi$ is a morphism of algebras. 
\end{itemize}
 On the other hand, the map $\psi : End_{RG}(R\Omega_{G})\to co\mu_{R}(G)$ is defined  as follows. Let $f \in End_{RG}(R\Omega_{G})$, then by Lemma \ref{lin_proj}, there exist a $G$-set $Z(f)$ over $\Omega_{G}\times \Omega_{G}$ such that $f = p_{L}(Z(f))$. Then $\psi$ is defined by: $$\psi(f) = p\circ \beta^{-1}(Z(f)).$$
 \begin{itemize}
 \item The map $\psi$ is well defined: if $Z$ is a $G$-set over $\Omega_{G}\times \Omega_{G}$ such that $p_{L}(Z)=0$, then we can express $Z$ in the usual basis of $RB(\Omega_{G}^2)$, that is the basis induced by the isomorphism $\beta$ and the usual basis of the Mackey algebra (see Proposition \ref{basis}). This basis is indexed by $H$ and $K$ two subgroups of $G$, an element $x$ of  the set of representatives of the double cosets $H\backslash G/K$ and a subgroup $L$ of $K\cap\ ^{x}H$ (up to conjugacy class). We denote by $I$ the set indexing this basis, and we denote by $Z_{H,K,L,x}$ the corresponding basis element. There are elements $\lambda_{H,K,L,x}$ of $R$ such that $Z= \sum_{I}\lambda_{H,K,L,x} Z_{H,K,L,x}$. Then $p_{L}(Z)=0$ if and only if for every $H$, $K$, we have:
\begin{equation*}
\sum_{L,x} \lambda_{H,K,L,x} p_{L}(Z_{H,K,L,x}) = 0.
\end{equation*} 
Let us recall the definition of $Z_{H,K,L,x}$:
\begin{equation*}
\xymatrix{
& G/L\ar[ld]_{\pi^{H}_{L}}\ar[dr]^{\pi^{K}_{L^x}\gamma_{L^x,x}} & \\
G/H & & G/K
}
\end{equation*}
But,
\begin{align*}
p_{L}(Z_{H,K,L,x}) &= \sum_{h\in [H/L]} ghxK \\&= |H\cap\ ^{x}K : L| \sum_{h \in [H/ H\cap\ ^{x}K]} ghxK\\ &= |H\cap\ ^{x}K : L| p_{L}(Z_{H,K,H\cap\ ^{x}K,x}).
\end{align*}
Moreover, the set of maps $p_{L}(Z_{H,K,H\cap\ ^{x}K,x})$ is, by Lemma \ref{lin_proj}, a basis set of $Hom_{RG}(RG/H,RG/K)$, so if $p_{L}(Z)=0$, we have, for $H$ and $K$ subgroups of $G$ and $x\in [H\backslash G/K]$:
\begin{equation*}
\sum_{L} |K\cap\ ^{x}H : L| \lambda_{H,K,L,x} = 0. 
\end{equation*}
Since in the cohomological Mackey algebra we have: $$p(t^{H}_{\ ^{x}L}xr^{K}_{L})=|K\cap\ ^{x}H : L|p(t^{H}_{H\cap\ ^{x}K} x r^{K}_{K\cap H^{x}}),$$ then, if $p_{L}(Z)=0$, we have $p\beta^{-1}(Z)=0$.
\item Since $\beta^{-1}$ and $p$ are two morphisms of algebras, the map $\psi$ is a morphism of algebras. 
\end{itemize}
The fact that $\phi$ and $\psi$ are two inverse isomorphisms follows from the fact that $\beta$ is an isomorphism. 
 \end{proof}
As immediate corollary, we have:
\begin{coro}
Let $G$ be a finite group and $R$ be a commutative ring with unit. 
Then, the set of $\overline{t^{H}_{H\cap\ ^{x}K} c_{K\cap H^x,x} r^{K}_{K\cap H^{x}}}$, when $H$ and $K$ run through the subgroups of $G$ and $x$ runs through a set of representatives of double cosets $H\backslash G/ K$ is an  $R$-basis of $co\mu_{R}(G)$. 
\end{coro}
\begin{proof}
This follows from the fact that this set is the image of the $R$-basis of $End_{RG}(R\Omega_{G})$ of Lemma \ref{lin_proj} introduced by Yoshida.
\end{proof}
 Now, using Dress' point of view, we have:
 \begin{theo}[Yoshida's Theorem for cohomological Mackey functors]\label{yoshida_functor}
 Let $G$ be a finite group and $R$ be a commutative ring with unit. We denote by $Fun_{R}(G)$ the category of $R$-linear contravariant functors from the category of finitely generated permutation $RG$-modules to the category of $R$-modules. Then
 \begin{equation*}
 Comack_{R}(G)\cong Fun_{R}(G).
 \end{equation*}
 \end{theo}
 \begin{proof}[Sketch of proof]
 This equivalence of categories can be constructed as follows: There is a Yoneda functor $Y$ from $Comack_{R}(G)$ to $Fun_{R}(G)$. More precisely, if $M$ is a cohomological Mackey functor, then $Y(M)$ is defined by: $$ Y(M) = Hom_{Comack_{R}(G)}( - , M) \circ FP_{-},$$ where $FP_{-}$ is the functor from the category of permutation $RG$-modules to the category of cohomological Mackey functors sending the $RG$-module $V$ to the fixed point functor $FP_{V}$. Here $FP_{V}$ is the Mackey functor defined by $$Hom_{RG}(-,V)\circ \Pi. $$ That is $FP_{V}(X)=Hom_{RG}(\Pi(X),V)$ for a finite $G$-set $X$.\\ On the other hand, if $F\in Fun_{R}(G)$, then $\Gamma$ is defined by: $\Gamma(F)=F\circ \Pi$.
 \newline Let us recall the units and co-units of the two pairs of adjoint functors $\Gamma\dashv Y$ and $Y\dashv\Gamma$. 
 \begin{itemize}
 \item For the adjunction $\Gamma\dashv Y$ we have: let $F$ be a functor of $Fun_{R}(G)$. The unit $\delta$ of this adjunction is the natural transformation defined by: let $V=RX$ be a permutation $RG$-module and $u\in F(RX)$. Let $Z$ be a finite $G$-set. Then,
 \begin{align*}
 \delta_{F}(V)(u)_{Z} : Hom_{RG}(RZ,R&X) \to F(RZ)\\
 & \alpha \mapsto F(\alpha)(u).
 \end{align*}
Let $M$ be a cohomological Mackey functor. The co-unit of this adjunction is the map $\epsilon_{M} : \Gamma \circ Y(M) \to M$ defined by: let $X$ be a finite $G$-set. Then,
\begin{align*}
\epsilon_{M}(X) : Hom_{Comack{R}(G)}(FP_{RX}&,M)\to M(X) \\
\alpha&\mapsto \alpha_{X}(Id_{RX}).
\end{align*}
\item For the second adjunction $Y\dashv\Gamma$, we have: let $F$ be a functor of $Fun_{R}^{+}(G)$ and let $M$ be a cohomological Mackey functor. Then the co-unit $\epsilon'$ of this adjunction is defined as follows. Let $X$ be a finite $G$-set. Then:
\begin{align*}
\epsilon'_{F}(X) : Hom_{Comack_{R}(G)}(FP_{RX}&,\Gamma(F))\to F(RX)\\
\phi &\mapsto \phi_{X}(Id_{RX}).
\end{align*}
For the unit it is a bit more complicate. Let $X$ and $Y$ be two finite $G$-sets. Let $m\in M(X)$. \\ Let $f_{Y}\in Hom_{RG}(Y,X)$. Then by Lemma \ref{lin_proj}, there exist a $G$-set $$(Y\overset{b}{\leftarrow} U \overset{a}{\rightarrow} X)$$ over $Y\times X$, denoted by $Z_{U,a,b}$, such that $$f_{Y} = p_{L}(Z_{U,a,b}).$$
The unit of this adjunction is:
\begin{align*}
\delta'_{M}(X) : M(&X) \to Hom_{Comack_{R}(G)}(FP_{X},M)\\
& m\mapsto \Big(  f_{Y} \mapsto M_{*}(b)\circ M^{*}(a)(m)\Big).
\end{align*}
Since $M$ is a cohomological Mackey functor, if $Z_{V,c,d}$ is another $G$-set over $Y\times X$ such that $f_{Y} = p_{L}(Z_{V,c,d})$, then $M_{*}(d)M^{*}(c)=M_{*}(b)M^{*}(a)$ (by the proof of Theorem \ref{yoshida_alg}), so the co-unit is well defined.
\end{itemize} 
\end{proof}
Let us denote by $perm_{R}(G)$ the full subcategory of $RG$-$Mod$ consisting of the finitely generated permutation $RG$-modules.
\begin{lemma}
The idempotent completion of $perm_{R}(G)$ is equivalent to the category of finitely generated permutation projective $RG$-modules. 
\end{lemma}
\begin{proof}
Let us denote temporarily by $\mathcal{A}$ the category of permutation projective $RG$-modules. Let $perm_{R}^{+}(G)$ be the idempotent completion of $perm_{R}(G)$.\\ The objects of this category are the pairs $(V,\pi)$ where $V$ is a permutation module and $\pi\in Hom_{perm_{R}(G)}(V,V)$ an idempotent. There is a natural functor $F$ from $perm_{R}^{+}(G)$ to $\mathcal{A}$ defined by $F(V,\pi)=\pi(V)$. This functor is dense and fully faithful. 
\end{proof}
We denote by $perm^{+}_{R}(G)$ the category of finitely generated permutation projective $RG$-modules and by $Fun^{+}_{R}(G)$ the category consisting of contravariant functors from $perm^{+}_{R}(G)$ to $R$-$Mod$. By general properties of the idempotent completion (\cite{groth_topos} Exemple 8.7.8 page 97.), the categories $Fun_{R}^{+}(G)$ and $Fun_{R}(G)$ are equivalent. So we have: 
\begin{equation*}
Comack_{R}(G)\cong Fun_{R}^{+}(G). 
\end{equation*}
We still denote by $Y\dashv \Gamma$ the equivalence after idempotent completion. 
\section{The center of the cohomological Mackey algebra.}
\begin{de}
Let $\mathcal{C}$ be a (small) additive category. The center of $\mathcal{C}$, denoted by $Z(\mathcal{C})$, is the endomorphism ring of the identity functor $Id_{\mathcal{C}}$ of the category~$\mathcal{C}$. 
\end{de}
It is well known that the definition of the center of a category is functorial in respect with the equivalences of categories. Since we were not able to find a reference for this fact, we sketch the proof. 
\begin{lemma}\label{fun_center}
Let $\mathcal{C}$ and $\mathcal{D}$ be two additive categories. Let $F\dashv G$ be an equivalence between $\mathcal{C}$ and $\mathcal{D}$. Then:
\begin{enumerate}
\item The functor $F$ induces a ring homomorphism from $Z(\mathcal{C})$ to $Z(\mathcal{D})$, denoted by $f$. 
\item The functor $G$ induces a ring homomorphism from $Z(\mathcal{D})$ to $Z(\mathcal{C})$, denoted by $g$.
\item The two homomorphisms $f$ and $g$ are inverse isomorphisms. 
\end{enumerate}
\end{lemma}
\begin{proof}
We denote by $\delta$ (resp. $\delta'$) the unit of the adjunction $G\dashv F$ (resp. $F\dashv G$) and by $\epsilon$ (resp. $\epsilon'$) the co-unit of the adjunction $G\dashv F$ (resp. $F \dashv G$), that is the following natural transformations:
\begin{align*}
\delta &: Id \to FG \\
\epsilon &: GF\to Id \\
\delta' &: Id\to GF\\
\epsilon' &: FG\to Id.
\end{align*}
Let $\eta$ be an endomorphism of $Id_{\mathcal{C}}$. Then, $f(\eta)$ is the natural transformation from the functor $Id_{\mathcal{D}}$ to himself defined as follows: if $D$ is an object of $\mathcal{D}$, then:
\begin{equation*}
f(\eta)_{D} = \epsilon'_{D}\circ F(\eta_{G(D)}) \circ \delta_{D} : D\to FG(D) \to FG(D)\to D.
\end{equation*}
Let $\gamma$  be an endomorphism of $Id_{\mathcal{D}}$. Then $g(\gamma)$ is the natural transformation defined as follows: if $C$ is an object of $\mathcal{C}$, then:
\begin{equation*}
g(\gamma)_{C} = \epsilon_{C} \circ G(\gamma_{G(C)}) \circ \delta'_{C}: C\to GF(C)\to GF(C)\to C.
\end{equation*}
\end{proof}
\begin{re}
The definition of the center of an additive category generalized the usual definition of the center of a ring. More precisely, if $R$ is a ring, then the center of the category $R$-$Mod$ is isomorphic to the center of the ring $R$ (see the proof of Proposition $2.2.7$ \cite{benson}).
\end{re}
\begin{prop}\label{form_center}
Let $G$ be a finite group and $R$ be a commutative ring with unit. Then, there is a ring isomorphism: $$\iota : Z(RG)\to Z(co\mu_{R}(G)).$$ If $z=\sum_{x\in G} \lambda_{x} x\in Z(RG)$, then
\begin{equation*}
\iota(z)=\sum_{H\leqslant G} \frac{1}{|H|} \sum_{x\in G}\lambda_{x} \overline{t^{H}_{1} c_{1,x} r^{H}_{1}}.  
\end{equation*}
Here, we denote by $\overline{x}$ the image of $x\in \mu_{R}(G)$ in the cohomological Mackey algebra. 
\end{prop}
\begin{proof}
The existence of an isomorphism between $Z(RG)$ and $Z(co\mu_{R}(G))$ is due to Bouc (Proposition $12.3.2$ of \cite{bouc_green}). It uses the point of view of Green Mackey functors. More precisely it is based on the fact that cohomological Mackey functors are modules over the Green functor $FP_{R}$ and the fact that the Yoshida algebra is isomorphic to $FP_{R}(\Omega_{G}\times \Omega_{G})$. \newline\indent Here, we give an elementary proof of this result, which allows us to specify an isomorphism. First we prove that $$Z(RG)\cong Z(End_{RG}(R\Omega_{G})).$$ Let $z\in Z(RG)$. Then the multiplication by $z$ on the $RG$-module $R\Omega_{G}$, denoted by $m_{z}(R\Omega_{G})$ is an element of the center of $End_{RG}(R\Omega_{G})$. \\ On the other hand, if $f\in End_{RG}(R\Omega_{G})$ is a central element, for $g\in End_{RG}(R\Omega_{G})$, the following diagram must commutes:
\begin{equation*}
\xymatrix{
R\Omega_{G}\ar[r]^{f}\ar[d]^{g} & R\Omega_{G}\ar[d]^{g}\\
R\Omega_{G}\ar[r]^{f} & R\Omega_{G}
}
\end{equation*}
By taking $g = R\Omega_{G}\twoheadrightarrow RG/H\to RG/H\hookrightarrow R\Omega_{G},$ we see that $f = \sum_{H\leqslant G} \widehat{f_{H,H}}$, where $f_{H,H} \in Z(End_{RG}(RG/H))$, and $\widehat{f_{H,H}}$ is the composite map: $$R\Omega_{G}\twoheadrightarrow RG/H\overset{f_{H,H}}{\to} RG/H\hookrightarrow R\Omega_{G}.$$
By taking $g = \widehat{(\pi_{1}^{H})_{*}}$, where $\pi_{1}^{H}$ is the natural projection $G/1\to G/H$, for $x\in RG/1$ we have $f_{H,H}((\pi^{H}_{1})_{*}(x)) = (\pi^{H}_{1})_{*}\big(f_{1,1}(x)\big)$. That is, if $x=\sum_{g\in [G/H]} \lambda_{g} gH\in RG/H$,
\begin{align*}
f_{H,H}(x) &=\sum_{g\in [G/H]} \lambda_{g}f_{H,H}(gH)\\
&=\sum_{g\in [G/H]} \lambda_{g}(\pi^{H}_{1})_{*}(gf_{1,1}(1)).
\end{align*}
But $f_{1,1}$ is  a central element of $End_{RG}(RG)$. So we have $gf_{1,1}(1)=f_{1,1}(1)g$. And for $x\in RG/H$ we have:
\begin{equation}\label{formule}
f_{H,H}(x) =\sum_{g\in [G/H]} \lambda_{g} f_{1,1}(1)gH = f_{1,1}(1)\cdot x.
\end{equation}
If $f\in Z(End_{RG}(R\Omega_{G}))$, then $z=f_{1,1}(1)\in Z(RG)$. By Formula (\ref{formule}), we have $$m_{f_{1,1}}(R\Omega_{G}) = f,$$ and it is clear that $m_{z}(R\Omega_{G})_{1,1}(1)=z$. 
\newline Since $co\mu_{R}(G)\cong End_{RG}(R\Omega_{G})$ the result of the lemma follows. Moreover, if $z\in Z(RG)$, then $z$ is a linear combination of elements of $G$, that is:
\begin{equation*}
z=\sum_{x\in G}\lambda_{x} x, 
\end{equation*}
where  $\lambda_{x} \in R$ for $x\in G$. 
\newline Now, let $Z_{H,H,x}$ be the following $G$-set over $G/H\times G/H$
\begin{equation*}
\xymatrix{
& G/1\ar[rd]^{\pi^{H}_{1}\gamma_{1,x}}\ar[dl]_{\pi^{1}_{H}}\\
G/H &&  G/H. 
}
\end{equation*}
Then, for $gH\in G/H$, we have:
\begin{align*}
\sum_{x\in G}\lambda_{x} p_{L}(Z_{H,H,x})(gH) &= \sum_{x\in G}\lambda_{x}\sum_{h\in H} ghxH\\
&=\sum_{h\in H} gh(\sum_{x\in G} \lambda_{x}x)H\\
&=\sum_{h\in H} g(\sum_{x\in G}\lambda_{x}x)hH\\
&=|H|(\sum_{x\in G} \lambda_{x}x)g.
\end{align*}
So, we have: $$m_{z}(R\Omega_{G}) = \sum_{H\leqslant G} \frac{1}{|H|} \sum_{x\in G}\lambda_{x} p_{L}(Z_{H,H,x}),$$
By the isomorphism of Theorem $\ref{yo1}$, this endomorphism of $R\Omega_{G}$ is sent to:
\begin{equation*}
\sum_{H\leqslant G} \frac{1}{|H|} \sum_{x\in G}\lambda_{x} \overline{t^{H}_{1} c_{1,x} r^{H}_{1}.} 
\end{equation*}
\end{proof}
\begin{re}
Since there are some denominators, it may not be clear that the formula of Proposition \ref{form_center} is defined for every ring. However if $z = \sum_{x\in G} \lambda_{x} x$, for $H\leqslant G$, we have:
\begin{align*}
 \sum_{x\in G}\lambda_{x} t^{H}_{1} c_{1,x} r^{H}_{1} = \sum_{x\in G}\lambda_{x} |H\cap\ ^{x}H| t^{H}_{H\cap\ ^xH} c_{H\cap H^{x},x} R^{H}_{H\cap H^{x}}.
\end{align*}
Here, in order to simplify the notations, for $x\in \mu_{R}(G)$, we still write $x$ the image of $x$ in $co\mu_{R}(G)$. Moreover we will denote by $x$ the map $c_{H\cap H^x,x}$. Since the basis element $t^{H}_{H\cap\ ^xH} xR^{H}_{H\cap H^x}$ depends only on the double coset $HxH$, we have:
\begin{align*}
 \sum_{x\in G}\lambda_{x} t^{H}_{1} c_{1,x} r^{H}_{1} = \sum_{g \in [H\backslash G/H]}t^{H}_{H\cap\ ^gH} g R^{H}_{H\cap H^g}\Big(\sum_{x\in O(g)} \lambda_{x}|H\cap\ ^{x}H|\Big),
\end{align*}
where $O(g)$ is the orbit of the element $g$ under the action $(h,h').g=hgh'$ for $h$ and $h'\in H$. So, we have:
\begin{align*}
 \sum_{x\in G}\lambda_{x} t^{H}_{1} c_{1,x} r^{H}_{1} = \sum_{g \in [H\backslash G/H]}t^{H}_{H\cap\ ^gH} g R^{H}_{H\cap H^g} |H\cap H^{g}|\Big(\sum_{(h,k)\in H\times H/H\cap\ ^{g}H} \lambda_{hgk}\Big)
\end{align*}
Now since $z$ is an element of the center, for every $h\in H$, we have $hz=zh$, so 
\begin{equation*}
\sum_{x\in G} \lambda_{h^{-1}x}x = \sum_{x\in G} \lambda_{xh^{-1}} x,
\end{equation*}
so for every $h\in H$, we have $\lambda_{xh}= \lambda_{hx}$. Then, we have:
\begin{equation*}
 \sum_{x\in G}\lambda_{x} t^{H}_{1} c_{1,x} r^{H}_{1} =  \sum_{g \in [H\backslash G/H]}t^{H}_{H\cap\ ^gH} g R^{H}_{H\cap H^g} |H\cap H^{g}|\frac{|H|}{|H\cap H^{g}|}\Big(\sum_{h\in H/H\cap\ ^{g}H} \lambda_{gh}\Big)
\end{equation*}
And finally, 
\begin{equation*}
\sum_{H\leqslant G}\frac{1}{|H|}\sum_{x\in G}\lambda_{x} t^{H}_{1} c_{1,x} r^{H}_{1} = \sum_{H\leqslant G} \sum_{g \in [H\backslash G/H]}t^{H}_{H\cap\ ^gH} g R^{H}_{H\cap H^g}\Big(\sum_{h\in H/H\cap\ ^{g}H} \lambda_{gh}\Big).
\end{equation*}
\end{re}
The formula of Propostion \ref{form_center} suggest the following definition. 
\begin{lemma}
Let $G$ be a finite group. Let $\rho : Z(RG) \to Z(Comack_{R}(G))$ be the map defined as follows. let $M$ be a cohomological Mackey functor. Let $X$ be a finite $G$-set and let $z$ be an element of $Z(RG)$. Let $m_{z}(RX)$ be the multiplication by $z$ on $RX$. There is a $G$-set $\big(X\overset{b}{\leftarrow} U \overset{a}{\rightarrow} X\big)$ over $X\times X$ denoted by $Z_{U,a,b}$ for some $G$-set $U$ such that $m_{z}(RX) = p_{L}(Z_{U,a,b})$. Then:
\begin{equation*}
\rho(z)_{M}(X) = M_{*}(b) M^{*}(a):M(X)\to M(X). 
\end{equation*}
\end{lemma}
\begin{proof}
Since $M$ is a cohomological Mackey functor this map does not depend on the choice of $Z_{U,a,b}$. Now, using the pullback property of $M$, it is easy to check that $\rho$ is a ring homomorphism. 
\end{proof}
\begin{re}\label{redr}
Since the categories $Comack_{R}(G)$ and $co\mu_{R}(G)$-$Mod$ are equivalent, their center are isomorphic. The reader familiar with the equivalence of categories between Dress' definition and Green's definition (or Th\'evenaz-Webb's definition) can see that the morphism $\rho$ is just the morphism induced by $\iota$ and the equivalence of categories between $Comack_{R}(G)$ and $co\mu_{R}(G)$-$Mod$.\end{re}
\begin{lemma}
Let $G$ be a finite group. Let $\eta : Z(RG) \to Z(Fun_{R}^{+}(G))$ be the map defined as follows. Let $z\in Z(RG)$ and let $F\in Fun_{R}^{+}(G)$. Let $V$ be a permutation projective module. Let us denote by $m_{z}$ the endomorphism of the identity functor of $RG$-$Mod$ corresponding to $z$. Then $\eta_z$ is the endomorphism of the identity functor of $Fun_{R}^{+}(G)$ defined by:
\begin{equation*}
\eta_z(F)_V = F(m_{z}(V)) : F(V)\to F(V).
\end{equation*}
The map $\eta$ is an ring homomorphism.
\end{lemma}
\begin{proof}
This is straightforward. 
\end{proof}
The Yoshida equivalence is compatible with the action of central idempotents: 
\begin{theo}[Yoshida Equivalence, block version]\label{yoshida}
There is a commutative diagram: 
\begin{equation}
\xymatrix{
& & & Z(Fun^{+}_{R}(G))\ar[dd]^{\cong}_{\gamma} \\
Z(co\mu_{R}(G))\ar[r]^{\iota} & Z(RG)\ar[urr]_{\eta}^{\cong}\ar[drr]_{\rho}^{\cong} & & \\
& & & Z(Comack_{R}(G))\ar@/^2pc/[lllu]^{\star}.
}
\end{equation}
Here, the map $\gamma$ is the ring homomorphism induced by the functor $\Gamma$ as in Lemma \ref{fun_center}. 
The arrow $\star$ is the map induced by the equivalence $Comack_{R}(G)$ and $co\mu_{R}(G)$-$Mod$ (see Remark \ref{redr}).
\newline Let $1=e+f\in Z(RG)$ be a decomposition of $1$ as a sum of two orthogonal idempotents. Then 
\begin{equation*}
Comack_{R}(G)\cong \rho(e)\big(Comack_{R}(G)\big)\oplus \rho(f)\big(Comack_{R}(G)\big).
\end{equation*}
and 
\begin{equation*}
Fun_{R}^{+}(G)=\eta(e)\big(Fun_{R}^{+}(G)\big)\oplus \eta(f)\big(Fun_{R}^{+}(G)\big).
\end{equation*}
If $b=e$ or $f$, then $\rho(b)\big(Comack_{R}(G)\big)\cong \eta(b)\big(Fun_{R}^{+}(G)\big)$.
\end{theo}
\begin{proof}[Sketch of proof]
Let $\sigma$ be a natural transformation of the identity functor of $Fun_{R}^{+}(G)$, let $Y_{RG}$ be the Yoneda functor $Hom_{RG}(-,RG)$, then $z_{\sigma}:=\bigg(\sigma_{Y_{RG}}(RG)(Id_{RG})\bigg)(1)$ is an element of $Z(RG)$. One can check that the map which sent $\sigma$ to $z_{\sigma}$ is the inverse isomorphism of $\eta$.
\newline Let $M$ be a cohomological Mackey functor in the sense of Dress. Let $z\in Z(RG)$, we denote by $m_{z}$ the corresponding natural transformation in $Z(RG$-$Mod)$.\newline If $f\in Z(Fun_{R}^{+}(G))$, then with the notations of Theorem \ref{yoshida}, we have:
\begin{equation*}
\gamma(f)_{M}= \epsilon_{M}\circ \Gamma(f_{Y(M)}) \circ \delta'_{M} : M\to \Gamma(Y(M)) \to \Gamma(Y(M)) \to M.
\end{equation*}
So, if $X$ is a finite $G$-set, and if $m\in M(X)$, we have:
\begin{align*}
\gamma_{M}(\eta_{z})(m) &= \delta'_{M}(X)\circ (m_{z}(X))\circ Id_{RX}\\
&=\delta'_{M}(X)(m_{z}(X))\\
\end{align*}
Now, if $m_{z}(X)= p_{L}(Z_{U,a,b})$, we have:
\begin{align*}
\gamma_{M}(\eta_{z})(m) = M_{*}(b)M^{*}(a)(m).
\end{align*}
This is equal to $\rho(z)_{M}(X)$.  
\end{proof}
Let $R$ be $\mathcal{O}$ or $k$, where $\mathcal{O}$ is a complete discrete valuation ring and $k$ is the residue field. Let $1=b_{1}+b_{2}+\cdots + b_{s}$ be a decomposition of $1$ in orthogonal sum of central primitive idempotent of $RG$. This decomposition induces a decomposition of $Comack_{R}(G)=\bigoplus_{i=1}^{s}\rho(b_i)Comack_{R}(G)$ and $Fun_{R}^{+}(G)=\bigoplus_{i=1}^{s}\eta(b_i)Fun_{R}^{+}(G)$. We have the following straightforward lemma:
\begin{lemma} Let $b$ be a block idempotent of $RG$. The category $\eta(b)(Fun_{R}^{+}(b))$ is equivalent to the category denoted by $Fun_{R}^{+}(b)$, consisting of contravariant functors from $perm^{+}_{R}(b)$ to $R$-$Mod$, where $perm^{+}_{R}(b)$ is the category consisting of the finitely generated {$p$-permutation} $RG$-modules which are in the block $RGb$.  
\end{lemma}
For a block $b$ of $RG$, we denote by $Comack_{R}(b)$ the category $\rho(b)Comack_{R}(G)$. 
\begin{coro}\label{yoshida_co}
Let $b$ be a block of $RG$. The we have:
\begin{equation*}
co\mu_{R}(G)\iota(b)\hbox{-}Mod \cong Comack_{R}(b) \cong Fun_{R}^{+}(b). 
\end{equation*}
\end{coro}
\begin{coro}\label{coro_block}
Let $P$ be a projective indecomposable cohomological Mackey functor. Then $P$ belongs to the block $Comack_{R}(b)$ if and only if $P(G/1)$ is an indecomposable $p$-permutation module in the block $RGb$.
\end{coro}
\begin{proof}
Let $P$ be a cohomological Mackey functor. Let us recall that, with Dress' notation $P(G/1)$ is an $RG$-module for the following action. Let $m\in P(G/1)$ and $x\in G$. Then $x.m = M^{*}(\gamma_{1,x})(m)$. The result follows from the fact that $(\rho(b)\cdot P)(G/1) = b\cdot P(G/1)$ and from Theorem $16.5$ \cite{tw} which says that $P(G/1)$ is a $p$-permutation module. In the other way, if $V$ is a $p$-permutation $RGb$-module, then $FP_{V}$ is a projective cohomological Mackey functor in $Comack_{R}(b)$.
\end{proof}
In the proof of Theorem $17.1$ of \cite{tw}, Th\'evenaz and Webb proved that the block of the category of the cohomological Mackey functors are in bijection with the block of $RG$. They defined the blocks of the category $Comack_{R}(G)$ using non-split short exact sequences between simple cohomological Mackey functors. Thanks to Corollary \ref{coro_block} and Proposition $16.10$ of \cite{tw} \big(in order to understand the projective cover of the simple cohomological Mackey functors\big), their block decomposition coincide with ours. 
\section{Permeable Morita equivalences.}
Let $R=\mathcal{O}$ or $k$ as above. With the version of Yoshida's equivalence of Corollary \ref{yoshida_co} it is not difficult to lift an equivalence between blocks of group algebras to an equivalence of the corresponding blocks of the cohomological Mackey algebras.
\begin{de}
Let $G$ and $H$ be two finite groups, let $b$ be a block of $RG$, let $c$ be a block of $RH$. A \emph{permable} $RHc$-$RGb$-bimodule is a bimodule $X$ such that: 
\begin{itemize}
\item[$\mathcal{P}:$]$X\otimes_{RGb}-$ is a functor from $perm^{+}_{R}(b)$ to $perm_{R}^{+}(c)$. 
\end{itemize}
\end{de}
\begin{lemma}\label{prop_fun}
Let $G$ and $H$ be two finite groups, let $b$ be a block of $RG$, let $c$ be a block of $RH$. Let $X$ be a \emph{permeable} $RHc$-$RGb$-bimodule.  
Then $X$ induces a functor, denoted by $\Phi_{X}: Comack_{R}(c)\to Comack_{R}(b)$ and defined in the proof. Moreover this functor sends an arbitrary fixed point functor to a fixed point functor.
\end{lemma}
\begin{proof}
We use the equivalence $Comack_{R}(b)\cong Fun^{+}_{R}(b)$ of Corollary \ref{yoshida_co}. One can define a functor $L_{X}$ from  $Fun^{+}_{R}(c)$ to $Fun^{+}_{R}(b)$ by $L_{X}(F)(V):=F(X\otimes_{RGb}V)$, for $F\in Fun_{R}^{+}(c)$ and $V\in perm_{R}^{+}(b)$. We denote by $\Phi_{X}$ the composite functor: 
\begin{equation*}
\xymatrix{
Fun_{R}^{+}(c)\ar[r]^{L_{X}} & Fun_{R}^{+}(b)\ar[d]^{\Gamma}\\
Comack_{R}(c)\ar[u]^{Y}\ar@{..>}[r]^{\Phi_{X}}&Comack_{R}(b)
}
\end{equation*}
so if $V$ is a $RHc$-module, and $Z$ is a finite $G$-set, then
\begin{align*}
\Phi_{X}(FP_{V})(Z)&=\Gamma(L_{X}(Y(FP_{V})))(Z)\\
&=Y(FP_{V})(X\otimes_{RGb}RZ)\\
&\cong Hom_{Comack_{R}(H)}(FP_{X\otimes_{RGb}RZ},FP_{V})\\
&\cong Hom_{RHc}(X\otimes_{RGb}RZ,V)\\
&\cong Hom_{RGb}(RZ,Hom_{RHc}(X,V))\\
&\cong FP_{Hom_{RHc}(X,V)}(Z). 
\end{align*}
This isomorphism is functorial in $Z$, so $\Phi_{X}(FP_{V})=FP_{Hom_{RHc}(X,V)}$.
\end{proof}
\begin{re}
This Lemma generalizes the construction defined by Bouc for permutation bimodules (see Section $3.12$ \cite{bouc_complex}).
\end{re}
\begin{de}
Let $G$ and $H$ be two finite groups, let $b$ be a block of $RG$ and $c$ be a block of $RH$. A \emph{permeable} (Morita) equivalence is an $RHc$-$RGb$-bimodule $X$ such that:
\begin{enumerate}
\item $X\otimes_{RGb}- : RGb$-$Mod\to$ $RHc$-$Mod$ is an equivalence of categories.
\item $X$ and $X^{*}:=Hom_{R}(X,R)$ are two permeable bimodules. 
\end{enumerate}
\end{de}
\begin{prop}\label{prop_morita}
Let $G$ and $H$ be two finite groups, let $b$ be a block of $RG$ and $c$ be a block of $RH$. Let $X$ be a permeable equivalence between $RGb$ and $RHc$. 
Then $Comack_{R}(b)\cong Comack_{R}(c)$. 
\end{prop}
\begin{proof}
By Lemma \ref{prop_fun}, we have a functor $L_{X}: Fun^{+}_{R}(c)\to Fun^{+}_{R}(b)$, and a functor \\ $L_{X^{*}}: Fun^{+}_{R}(b)\to Fun^{+}_{R}(c)$. It is clear that these two functors are two quasi-inverse equivalences between $Fun_{R}^{+}(c)$ and $Fun_{R}^{+}(b)$. 
\end{proof}
\begin{re}
One may ask if there exist \emph{permeable} Morita equivalences. Let $G$ be a finite group, and $P$ be a Sylow $p$-subgroup of $G$ and $H$ be its normalizer. Let $b$ be a block of $kG$ and with defect group $P$ and let $c$ be the Brauer correspondent of this block in $N_{G}(P)$. If $kGb$-$Mod$ is Morita equivalent to $kN_{G}(P)c$-$Mod$ by a $p$-permutation bimodule (that is a `splendid' Morita equivalence) then the two conditions are satisfied. 
\end{re}
\begin{re}
There exist $RG$-$RH$-bimodules which are \emph{not} $p$-permutation bimodules but which are \emph{permeable}. The most radical example is for $G=H=C_{2}$ and $R=\overline{\mathbb{F}_{2}}$. Then, all the $RG$-modules are permutation modules. So every $R[C_{2}\times C_{2}]$-module induces a functor between $perm_{R}(H)$ and $perm_{R}(G)$, and there are infinitely many isomorphism classes of $R[C_{2}\times C_{2}]$-modules and only $5$ isomorphism classes of permutation $R[C_{2}\times C_{2}]$-modules. Moreover, there are examples of Morita equivalences between blocks of group algebras which are \emph{not} `splendid' but which are permeable. The easiest example is probably for the self equivalences of $kC_{3}$ when $k=\mathbb{F}_{3}$. Indeed there are two permutations bimodules inducing a self-Morita equivalence of $kC_{3}$ and $6$ isomorphism classes of self-Morita equivalence of $kC_{3}$. This follows from elementary results on the Picard group of a basic $k$-algebra and easy computations. Now all of these $6$ equivalences are permeable. 
\end{re}
\section{Derived equivalences between blocks of cohomological Mackey algebras.}
Let $G$ and $H$ be two finite groups. Let $R=\mathcal{O}$ or $k$. Let $b$ be a block of $RG$ and $c$ be a block of $RH$. In this section, we prove that one can lift a derived equivalence between blocks of group algebras into a derived equivalence between the corresponding blocks of cohomological algebras as soon as this derived equivalence respects $p$-permutation modules. Since this part is rather technical, we fix the notations. 
\begin{notations}
\begin{itemize}
\item Let $X$ be an $RHc$-$RGb$-bimodule, then we denote by $t_{X}$ the functor from $RGb$-$Mod$ to $RHc$-$Mod$ induced by the tensor product with $X$.\\ If $f : X\to Y$ is a morphism of $RGb$-$RHc$-bimodules, we denote by $\hat{f}$ the natural transformation between the functors $t_{X}$ and $t_{Y}$.
\item Let $X$ be an $RHc$-$RGb$-bimodule such that $t_{X}$ induces a functor from $perm_{R}^{+}(b)$ to $perm_{R}^{+}(c)$. Let $F$ be a functor of $Fun_{R}^{+}(c)$, we can precompose the functor $F$ by the functor $t_{X}$, this gives a functor $F\circ t_{X}$ of the category $Fun_{R}^{+}(b)$. We will denote by $F\widetilde{X}$ this functor. 
\item Let $(F_{\bullet}, \eta_{\bullet})$ be a complex of functors of $Fun_{R}^{+}(c)$. We choose to label the complex by decreasing order, that is $\eta_{i}$ is a natural transformation from the functor $F_{i}$ to the functor $F_{i-1}$. 
\item If $(X_{\bullet},d_{\bullet})$ is a complex (written in decreasing order) of \emph{permeable} $RHc$-$RGb$-bimodules, then $((t_{X})_{\bullet}, \hat{d}_{\bullet})$ is a complex of functors from $RGb$-$Mod$ to $RHc$-$Mod$. 
\end{itemize}
\end{notations}
Let $(F_{\bullet}, \eta_{\bullet})$ be such a complex of functors and let $(X_{\bullet},d_{\bullet})$ be a complex of permeable $RHc$-$RGb$-bimodules. Then we can precompose the complex $F_{\bullet}$ by the complex of functors $(t_{X})_{\bullet}$. This gives a double complex: 
\begin{equation}\label{dc}
\xymatrix{
& \vdots& &\vdots&\\
\cdots\ar[r]& F_{i}\widetilde{X_{j}}\ar[rr]^{\eta_{i}\widetilde{X_{j}}}\ar[u] && F_{i-1}\tilde{X_{j}}\ar[r]\ar[u] &\cdots\\
\cdots\ar[r]& F_{i} \widetilde{X_{j-1}}\ar[rr]^{\eta_{i}\widetilde{X_{j-1}}}\ar[u]^{(-1)^{i}F_{i}\tilde{d_{j}}} && F_{i-1}\widetilde{X_{j-1}}\ar[u]_{(-1)^{i-1}F_{i-1}\widetilde{d_{j}}}\ar[r]&\cdots\\
& \vdots\ar[u] && \vdots\ar[u] &
}
\end{equation}
Here, we use the following notations:
\begin{enumerate}
\item Let $d : X\to Y$ be a map between two $RHc$-$RGb$-bimodules and let $F$ be a functor of the category $Fun_{R}^{+}(c)$.Then $F\tilde{d}$ is the natural transformation from $F\tilde{Y}$ to $F\tilde{X}$ defined by: if $M$ is a $p$-permutation $RGb$-module, then
\begin{equation*}
F\tilde{d}(M)=F(d\otimes Id_{M}) : F(Y\otimes_{RGb}M)\to F(X\otimes_{RGb}M). 
\end{equation*}
\item Let $\eta$ be a natural transformation from $F$ to $C$, where $F$ and $C$ belong to $Fun_{R}^{+}(c)$. Let $X$ be a permeable $RGb$-$RHc$-bimodule. Then $\eta\tilde{X}$ is the natural transformation from $F\tilde{X}$ to $C\tilde{X}$ defined by: let $M$ be a $p$-permutation $RGb$-module. Then
\begin{equation*}
\eta\tilde{X}(M)=\eta(X\otimes_{RGb}M) : F(X\otimes_{RGb}M)\to C(X\otimes_{RGb}M).
\end{equation*}
\end{enumerate}
Let $(F_{\bullet}, \eta_{\bullet})$ be a complex of functors which belong to $Fun_{R}^{+}(c)$. Let $(X_{\bullet},d_{\bullet})$ be a complex of \emph{permeable} $RHc$-$RGb$-bimodules. Then we denote by $(L_{X_{\bullet}}(F_{\bullet}),\delta_{\bullet})$ the total complex of the double complex (\ref{dc}), that is:
\begin{equation*}
\big(L_{X_{\bullet}}(F_{\bullet})\big)_{k} = \bigoplus_{i-j=k} F_{i}\widetilde{X_{j}},
\end{equation*} 
and the differential is given by the family of natural transformations $\delta_{k}$ defined by $$\delta_{k}=\bigoplus_{i-j=k} (-1)^{i}F_{i}\widetilde{d_{j+1}}+\eta_{i}\widetilde{X_{j}}.$$  More explicitly, let $M$ be a $p$-permutation $RGb$-module.\\ Let $w=(w_{i,j})_{i-j=k} \in \bigoplus_{i-j=k} F_{i}(X_{j}\otimes_{RGb}M)$. Then $\delta_{k}(M)=\bigoplus_{i-j=k}\delta_{i,j}(M)$, where:
\begin{equation*}
\delta_{i,j}(M)(w_{i,j})=(-1)^{i}F_{i}(d_{j+1}\otimes Id_{M})(w_{i,j}) + \eta_{i}(X_{j}\otimes_{RGb} M)(w_{i,j}). 
\end{equation*}
Here, we use the notation $w_{i,j}$ which is the projection of $w$ on the composant $F_{i}(X_{j}\otimes_{RGb} M)$. 
\begin{lemma}
With the previous notations,
\begin{enumerate}
\item $(L_{X_{\bullet}}(F_{\bullet}),\delta_{\bullet})$ is a complex.
\item $F_{\bullet} \mapsto L_{X\bullet}(F_{\bullet})$ is an additive functor from the category $Ch^{-}(Fun_{R}^{+}(c))$ to the category $Ch^{-}(Fun_{R}^{+}(b))$. 
\item The functor $F_{\bullet} \mapsto L_{X\bullet}(F_{\bullet})$  induces a triangulated functor between the corresponding homotopy categories. 
\end{enumerate}
\end{lemma}
\begin{proof}
\begin{enumerate}
\item Let $k$ be an integer. We have to check that $\delta_{k-1}\circ \delta_{k}=0$. Let $M$ be a $p$-permutation $RGb$-module and let $w=(w_{i,j})_{i-j=k} \in \big(L_{X_{\bullet}}(F_{\bullet})\big)_{k}$. It is enough to see that the $\big(\delta_{k-1}(M)\circ \delta_{k}(M)\big)_{s,t}=0$, where this is the projection of $\delta_{k-1}(M)\circ\delta_{k}(M)$ on the composant $F_{s}(X_{t}\otimes M)$ for $s-t=k-2$. \\
Then, we have for $s-t=k-2$:
{\small\begin{align*}
(\delta_{k-1}(M)\circ \delta_{k}(M)(w))_{s,t}& = \eta_{s+1}(X_{t}\otimes_{RGb}M)\big((\delta_{k}(w))_{s+1,t}\big) \\&+ (-1)^s F_{s}(d_{t}\otimes Id_{M})\big((\delta_{k}(w))_{s,t-1}\big) \\
&=\eta_{s+1}(X_{t}\otimes_{RGb}M)\big(\eta_{s+2}(X_{t}\otimes_{RGb}M)(w_{s+2,t})\big) \\
&+(-1)^{s+1}\eta_{s+1}(X_{t}\otimes_{RGb}M)\big( F_{s+1}(d_{t}\otimes Id_{M})(w_{s+1,t-1})\big)\\
&+(-1)^{s}F_{s}(d_{t}\otimes Id_{M})\big(\eta_{s+1}(X_{t-1}\otimes_{RGb}M)(w_{s+1,t-1})\big)\\
&+F_{s}(d_{t}\otimes Id_{M})\big( F_{s}(d_{t-1}\otimes Id_{M} )(w_{s,t-2})\big),
\end{align*}}
but $\eta$ is a differential for the complex $F_{\bullet}$ and $d_{\bullet}$ is a differential for the complex $X_{\bullet}$. So, we have:
\begin{align*}
(\delta_{k-1}(M)\circ \delta_{k}(M)(w))_{s,t}&=(-1)^{s}F_{s}(d_{t}\otimes Id_{M})\big(\eta_{s+1}(X_{t-1}\otimes_{RGb}M)(w_{s+1,t-1})\big) \\
& + (-1)^{s+1}\eta_{s+1}(X_{t}\otimes_{RGb}M)\big( F_{s+1}(d_{t}\otimes Id_{M})(w_{s+1,t-1})\big).
\end{align*}
Since $\eta_{s+1}$ is a natural transformation from $F_{s+1}$ to $F_{s}$, the following diagram is commutative:
\begin{equation*}
\xymatrix{
F_{s+1}(X_{t}\otimes_{RGb} M)\ar[rrr]^{\eta_{s+1}(X_{t}\otimes_{RGb}M)} &&& F_{s}(X_{t}\otimes_{RGb}M) \\
F_{s+1}(X_{t-1}\otimes_{RGb}M)\ar[u]^{F_{s+1}(d_{t}\otimes Id_{M})}\ar[rrr]^{\eta_{s+1}(X_{t-1}\otimes_{RGb}M)} &&& F_{s}(X_{t-1}\otimes_{RGb}M)\ar[u]_{F_{s}(d_{t}\otimes Id_{M})} 
}
\end{equation*}
This proves that $\delta_{\bullet}$ is actually a differential. 
\item Let $(F_{\bullet},\eta_{\bullet})$ and $(C_{\bullet},\gamma_{\bullet})$ be two complexes of functors which belong to $Fun_{R}^{+}(c)$. Let $\phi=(\phi_{\bullet})$ be a morphism from $(F_{\bullet},\eta_{\bullet})$ to $(C_{\bullet},\gamma_{\bullet})$. One may define a natural transformation $\Phi_{k}$ from $\big(L_{X_{\bullet}}(F_\bullet)\big)_{k}$ to $\big(L_{X_{\bullet}}(C_{\bullet})\big)$ by: $\Phi_{k}:=\bigoplus_{i-j} \phi_{i}\widetilde{X_{j}}$, where $\phi_{i}\widetilde{X_{j}}$ is the natural transformation from $F_{i}\widetilde{X_{j}}$ to $C_{i}\widetilde{X_{j}}$ defined as follows: if $M$ is a $p$-permutation $RGb$-module, then
\begin{equation*}
\phi_{i}\widetilde{X_{j}}(M)=\phi(X_{j}\otimes_{RGb}M) : F_{i}(X_{j}\otimes_{RGb}M)\to C_{i}(X_{j}\otimes_{RGb}M). 
\end{equation*}
We have to check that $(\Phi_{k})_{k\in\mathbb{Z}}$ is a morphism of complexes, i-e, we have to check that $\Phi$ commutes with the differentials.\newline We denote, here, by $\delta_{\bullet}$ the differential of $L_{X_{\bullet}}(F_{\bullet})$ and $\Delta_{\bullet}$ the differential  of $L_{X_{\bullet}}(C_{\bullet})$. Let $w\in \bigoplus_{i-j=k} F_{i}(X_{j}\otimes_{RGb} M)$. Then for $s-t=k-1$, we have:
\begin{align*}
\big((\Phi_{k-1}(M)\circ \delta_{k}(M))(w)\big)_{s,t}&=\phi_{s}(X_{t}\otimes_{RGb}M)\big(\delta_{k}(w)_{s,t}\big)\\
&=\phi_{s}(X_{t}\otimes_{RGb}M)\big(\eta_{s+1}(X_{t}\otimes M)(w_{s+1,t})\big)\\
&+ (-1)^{s} \phi_{s}(X_{t}\otimes_{RGb}M)\big(F_{s}(d_{t}\otimes Id_{M})(w_{s,t-1})\big). 
\end{align*}
On the other hand, we have:
\begin{align*}
\big(\Delta_{k}(M)\circ \Phi_{k}(M)(w)\big)_{s,t} &= \gamma_{s+1}(X_{t}\otimes M)\big(\phi_{s+1}(X_{t}\otimes_{RGb}M)(w_{s+1,t})\big) \\
&+ (-1)^{s} C_{s}(d_{s}\otimes Id_{M})\big(\phi_{s}(X_{t-1}\otimes_{RGb}M)(w_{s,t-1})\big).
\end{align*}
So, the fact that $\Phi_{\bullet}$ is a morphism of complexes follows from the commutativity of these two diagrams:
\begin{equation*}
\xymatrix{
F_{s+1}(X_{t}\otimes_{RGb}M) \ar[rrr]^{\eta_{s+1}(X_{t}\otimes_{RGb}M)} \ar[d]^{\phi_{s+1}(X_{t}\otimes_{RGb}M)} &&& F_{s}(X_{t}\otimes_{RGb} M)\ar[d]^{\phi_{s}(X_{t}\otimes_{RGb}M)}\\
C_{s+1}(X_{t}\otimes_{RGb}M) \ar[rrr]^{\gamma_{s+1}(X_{t}\otimes_{RGb}M)} &&& C_{s}(X_{t}\otimes_{RGb} M)
}
\end{equation*}
Here, the commutativity follows from the fact that $\phi_{\bullet}$ is a morphism of complexes. 
\begin{equation*}
\xymatrix{
F_{s}(X_{t-1}\otimes_{RGb}M)\ar[rrr]^{F_{s}(d_{t}\otimes Id_{M})} \ar[d]^{\phi_{s}(X_{t-1}\otimes_{RGb}M)} & & & F_{s}(X_{t}\otimes_{RGb}M)\ar[d]^{\phi_{s}(X_{t}\otimes_{RGb}M)}  \\
C_{s}(X_{t-1}\otimes_{RGb}M)\ar[rrr]^{C_{s}(d_{t}\otimes Id_{M})} & & & C_{s}(X_{t}\otimes_{RGb}M) 
}
\end{equation*}
Here, the commutativity comes from the fact that $\phi_{s}$ is a natural transformation from $F_{s}$ to $C_{s}$. It is now clear that $L_{X_{\bullet}}$ is an additive functor, and we denote by $L_{X_{\bullet}}(\phi)$ the family of natural transformations $\Phi_{\bullet}$. 
\item Since the functor $L_{X}$ is additive, it induces a functor between the corresponding homotopy categories. 
It remains to see that the functor $L_{X}$ is triangulated. Let $(F_{\bullet},\eta_{\bullet})$ and $(C_{\bullet},\gamma_{\bullet})$ be two complexes of functors which belong to $Fun_{R}^{+}(c)$. Let $f$ be a morphism between these two complexes. We need to check that $L_{X}(cone(f))\cong cone(L_{X}(f))$. We use the following notations:
\begin{itemize}
\item The differential  of $cone(f)$ is denoted by $\beta$.
\item The differential of $L_{X_{\bullet}}(F_{\bullet})$ is denoted by $\delta$.
\item The differential of $L_{X_{\bullet}}(C_{\bullet})$ is denoted by $\Delta$.
\item The differential of $L_{X_{\bullet}}(cone(f))$ is denoted by $\partial$.
\item The differential of $cone(L_{X}(f))$ is denoted by $D$. 
\end{itemize}
Recall that (\cite{weibel} Section 1.5) the mapping cone of $f$ is defined as follow:
\begin{equation*}
cone(f)_{k} = F_{k-1}\oplus C_{k},
\end{equation*}
and the differential is the natural transformation from $cone(f)_{k}$ to $cone(f)_{k-1}$ defined by the following diagram: 
\begin{equation*}
\xymatrix{
F_{k-1}\ar[rr]^{-\eta_{k-1}}\ar@{}[d]^{\oplus}\ar[rrd]^{-f_{k-1}} && F_{k-2}\ar@{}^{\oplus}[d]\\
C_{k}\ar[rr]^{\gamma_{k}} && C_{k-1}
}
\end{equation*}
So, the differential $\partial_{k}$ from $cone(L_{X_{\bullet}}(f))_{k}$ to $cone(L_{X_{\bullet}}(f))_{k-1}$ is the natural transformation defined by $-\delta_{k-1} - L_{X_{\bullet}}(f)_{k-1} + \Delta_{k}$. On the other hand, \begin{align*}L_{X_{\bullet}}(cone(f))_{k} &= \bigoplus_{i-j = k} cone(f)_{i}\tilde{X_{j}} \\
&=\bigoplus_{i-j=k}F_{i-1}\tilde{X_{j}} \bigoplus_{i-j=k} C_{i}\tilde{X_{j}} \\
&= L_{X_{\bullet}}(F)_{k-1} \oplus L_{X_{\bullet}}(C)_{k}. 
\end{align*}
Let $M$ be a $p$-permutation $RGb$-module. Let $w\in cone\Big(L_{X_{{\bullet}}}(f)(M)\Big)_{k}$, in order to compute the differential of this complex, we denote by $w^{F}$ the projection of the element $w$ on $\bigoplus_{i-j=k-1}F_{i}(X_{j}\otimes_{RGb}M)$, and $w^{C}$ the projection on $\bigoplus_{i-j=k}C_{i}(X_{j}\otimes_{RGb}M)$. Let $s$ and $t$ be integers such that $s-t=k-1$. Then the projection of $D_{k}(w)$ on $cone(f)_{s,t}$ is:
\begin{align*}
\big(D_{k}(w)\big)_{s,t} &= \beta_{s+1}\widetilde{X_{t}}(w_{s+1,t}) + (-1)^{s}cone(f)_{s}\widetilde{d_{t}}(w_{s,t-1})\\
&=(-1)^{s}F_{s-1}\widetilde{d_{t}}((w^{F})_{s-1,t-1}) - \eta_{s}\widetilde{X_{t}}((w^{F})_{s,t})\\
&+ \gamma_{s+1}\widetilde{X_{t}}((w^{C})_{s+1,t}) + (-1)^{s} C_{s}(w^{C}_{s,t-1}) - f_{s}\widetilde{X_{t}}((w^{F})_{s,t})\\
&= \big(-\delta_{k-1}(w^{F})\big)_{s-1,t} + \big(\Delta_{k}(w^{C})\big)_{s,t} - \big(L_{X_{\bullet}}(f)(w^{F})\big)_{s,t}\\
&= \partial_k(w)_{s,t}. 
\end{align*}
Recall, that the exact triangles in the homotopy category are given by the triangles which are isomorphic to:
\begin{equation*}
F\overset{f}{\rightarrow} C \to cone(f) \to F[1].
\end{equation*}
Here the map from $C$ to $cone(f)$ (denoted by $i$) is the injection of $C$ in $cone(f)$ and the map (denoted by $p$) from $cone(f)$ to $F[1]$ is given by the projection of $F_{i-1} \subset cone(f)_{i}$ on $F[1]_{i}$ (see \cite{weibel} 1.52). 
\newline It is clear that $L_{X_{\bullet}}(F[1])=(L_{X_{\bullet}}(F))[1]$, moreover it is clear that $L_{X}(i)$ is the injection of $L_{X_{\bullet}}(C)$ in $cone(L_{X_{\bullet}}(f))=L_{X_{\bullet}}(cone(f))$ and $L_{X_{\bullet}}(p)$ is the projection of $cone(L_{X_{\bullet}}(f))=L_{X_{\bullet}}(cone(f))$ on $L_{X_{\bullet}}(F)[1]$. So
\begin{equation*}
L_{X_{\bullet}}(F) \overset{L_{X_{\bullet}}(f)}{\rightarrow} L_{X_{\bullet}}(C) \to L_{X_{\bullet}}(cone(f)) \to (L_{X_{\bullet}}(F))[1],
\end{equation*}
is an exact triangle. 
\end{enumerate}
\end{proof}
\begin{lemma}
Let $X_{\bullet}$ and $Y_{\bullet}$ be two bounded complexes of permeable $RHc$-$RGb$-bimodules. Then:
\begin{enumerate}
\item The two functors $L_{X_{\bullet}\oplus Y_{\bullet}}$ and $L_{X_{\bullet}}\oplus L_{Y_{\bullet}}$ are isomorphic as functors from $K^{-}(Fun_{R}^+(c))$ to $K^{-}(Fun_{R}^{+}(b))$.
\item If the complex $X_{\bullet}$ is contractible, then the functor $L_{X_{\bullet}}$ is contractible in the following sense: the complex $L_{X_{\bullet}}(F_{\bullet})$ is (naturally in $F$) contractible for every complex $F_{\bullet}$ of functors which belong to $Fun_{R}^{+}(c)$.
\end{enumerate}
\begin{proof}
\begin{enumerate}
\item Let $(F_{\bullet},\eta_{\bullet})$ and $(C_{\bullet},\gamma_{\bullet})$ be two complexes of functors which belong to $Fun_{R}^{+}(c)$. Let $f : F_{\bullet} \to C_{\bullet}$ be a morphism between theses two complexes. It is clear that $L_{X_{\bullet}\oplus Y_{\bullet}}(F_{\bullet})\cong L_{X_{\bullet}}(F_{\bullet})\oplus L_{Y_{\bullet}}(F_{\bullet})$. Let $M$ be a $p$-permutation $RGb$-module, let $j$ be an integer. We denote by $\zeta_{M,j}$ the composite: 
$$X_{j} \otimes_{RGb} M \to X_{j}\otimes_{RGb} M \oplus Y_{j}\otimes_{RGb} M \cong (X_{j}\oplus Y_{j})\otimes_{RGb}M.$$
The functoriality of the isomorphism follows from the fact that, for $i,j \in \mathbb{Z}$, the following diagrams (and the corresponding diagrams for the terms of $Y_{\bullet}$) are commutative:
\begin{equation*}
\xymatrix{
F_{i}((X_{j}\oplus Y_{j})\otimes_{RGb}M)\ar[rrr]^{F_{i}(\zeta_{M,j})}\ar[d]_{f_{i}((X_{j}\oplus Y_{j})\otimes_{RGb}M)} & &&F_{i}(X_{j}\otimes M)\ar[d]^{f_{i}(X_{j}\otimes M)} \\
C_{i}((X_{j}\oplus Y_{j})\otimes_{RGb}M)\ar[rrr]^{C_{i}(\zeta_{M,j})} &&& C_{i}(X_{j}\otimes M)\\
}
\end{equation*}
\item Let $X_{\bullet}$ be a contractible two-sided bounded complex. That is, there is a family of maps $s=(s_{j})_{j\in \mathbb{Z}}$, where $s_{j}$ is a map from  $X_{j}$ to  $X_{j+1}$, such that we have for $j\in\mathbb{Z}$:
\begin{equation*}
Id_{X_{j}} = s_{j-1}d_{j} + d_{j+1}s_{j}. 
\end{equation*}
Let $(F_{\bullet}, \eta_{\bullet})$ be a complex of functors which belong to $Fun_{R}^{+}(c)$. Then one can defined a family $(F\widetilde{s})$ of natural transformations $(F\widetilde{s})_{k}$ from $L_{X_\bullet}(F_\bullet)_{k}$ to $L_{X_\bullet}(F_\bullet)_{k+1}$ as: $$(F\widetilde{s})_{k}= \bigoplus_{i-j=k} (-1)^{i}F_{i}\widetilde{s_{j-1}},$$ where $F_{i}\widetilde{s_{j-1}}$ is the natural transformation defined as: let $M$ be a $p$-permutation $RGb$-module. Then
\begin{align*}
F_{i}\widetilde{s_{j-1}}(M)=F_{i}(s_{j-1}\otimes_{RGb} Id_{M}) : F_{i}(X_{j}\otimes_{RGb}M) \to F_{i}(X_{j-1}\otimes_{RGb}M). 
\end{align*}
Now, we have to check that $Id_{L_{X_\bullet}(F_\bullet)} = \delta_{k+1} F\widetilde{s}_{k} + F\widetilde{s}_{k-1} \delta_{k}$. Let $i$ and $j$ be two integers such that $i-j=k$. If $w\in(L_{X_\bullet}(F_\bullet)(M))_{k}$, then we have:
\begin{align*}
\big(\delta_{k+1}F\widetilde{s}_{k}(w)\big)_{i,j} &= \eta_{i+1}\widetilde{X_{j}}\big((F\widetilde{s}_{k}(w))_{i+1,j}\big) + (-1)^{i} F_{i}\widetilde{dj}\big((F\widetilde{s}_{k}(w))_{i,j-1}\big)\\
&=(-1)^{i+1}\eta_{i+1}\widetilde{X_{j}} F_{i+1}\widetilde{s_{j}}(w_{i+1,j+1})\\
&+F_{i}(s_{j-1}d_{j}\otimes_{RGb}Id_{M})(w_{i,j}). 
\end{align*}
On the other hand, we have:
\begin{align*}
\big(F\widetilde{s}_{k-1}(\delta_{k}(w))\big)_{i,j} &= (-1)^{i}F_{i}(s_{j}\otimes_{RGb}Id_{M})\big((\delta_{k}(w))_{i,j+1}\big)\\
&= (-1)^{i}F_{i}(s_{j}\otimes_{RGb}Id_{M})\eta_{i+1}\widetilde{X_{j+1}}(w_{i+1,j+1}) \\
&+ F_{i}(d_{j+1}s_{j}\otimes_{RGb}Id_{M})(w_{i,j}).  
\end{align*}
The result follows from the commutativity of the next diagram:
\begin{equation*}
\xymatrix{
F_{i+1}(X_{j+1}\otimes_{RGb}M) \ar[rrr]^{F_{i+1}(s_{j}\otimes_{RGb}Id_{M})}\ar[d]_{\eta_{i+1}(X_{j+1}\otimes_{RGb}M)} &&& F_{i+1}(X_{j}\otimes_{RGb}M)\ar[d]_{\eta_{i+1}(X_{j}\otimes_{RGb}M)}\\
F_{i}(X_{j+1}\otimes_{RGb}M) \ar[rrr]^{F_{i}(s_{j}\otimes_{RGb}Id_{M})} &&& F_{i}(X_{j}\otimes_{RGb}M)
}
\end{equation*}
Moreover, this construction is functorial in $F$, so the functor $L_{X_{\bullet}}$ is isomorphic to the the zero functor from $K^{-}(Fun_{R}^{+}(c))$ to $K^{-}(Fun_{R}^{+}(b))$ when $X_{\bullet}$ is contractible. 
\end{enumerate}
\end{proof}
\end{lemma}
\begin{lemma}
Let $G$, $H$ and $K$ be finite groups. Let $b$ be a block of $RG$, let $c$ be a block of $RH$ and let $d$ be a block of $RK$. Let $(X_{\bullet},d^{X}_{\bullet})$ be a bounded complex of permeable $RHb$-$RGc$-bimodules. Let $(Y_\bullet,d^{Y}_{\bullet})$ be a bounded complex of permeable $RKd$-$RHc$-bimodules. Then, we have an isomorphism of functors:
\begin{equation*}
L_{X}\circ L_{Y} \cong L_{Y\otimes_{RHc} X} 
\end{equation*}
\end{lemma}
\begin{proof}
We use the following convention for the tensor product of complexes:
\begin{align*}
Y_{\bullet}\otimes_{RHc}X_{\bullet} = \bigoplus_{i+j=k} Y_{i}\otimes_{RHc} X_{j},
\end{align*}
the differential, denoted by $D_{\bullet}$ is:
\begin{equation*}
D_{k}=\bigoplus_{i+j=k}\big((-1)^{i}Id_{Y_{i}}\otimes d^{X}_{j} + d_{i}^{Y}\otimes Id_{X_{j}}\big). 
\end{equation*}
Let $M$ be a $p$-permutation $RGb$-module and let $k$ be an integer. Let $F_{\bullet}$ be a complex of functors which belong to $Fun_{R}^{+}(d)$. Since the functors $F_{i}$ are additive functors, and since $X_{\bullet}$ and $Y_{\bullet}$ are bounded complex, it is clear that: $$L_{X_{\bullet}} \circ L_{Y_{\bullet}}(F)(M)_{k} \cong L_{Y_{\bullet}\otimes_{RHc}X_{\bullet}}(F)(M)_{k}.$$ Indeed:
\begin{align*}
L_{X_{\bullet}} \circ L_{Y_{\bullet}}(F)(M)_{k} &= \bigoplus_{n\in\mathbb{Z}}(L_{Y}(F))_{n}(X_{n-k}\otimes_{RGb}M)\\
&=\bigoplus_{n\in\mathbb{Z}} \bigoplus_{m\in\mathbb{Z}}(F_{m}(Y_{m-n}\otimes_{RHc}X_{n-k}\otimes_{RGb}M))\\
&=\bigoplus_{m\in\mathbb{Z}} F_{m}\big(\bigoplus_{n\in\mathbb{Z}}Y_{m-n}\otimes_{RHc}X_{n-k}\otimes_{RGb}M\big)\\
&=\bigoplus_{m\in\mathbb{Z}} F_{m}\big((Y_{\bullet} \otimes_{RHc} X_{\bullet})_{m-k}\otimes_{RGb}M\big)\\
&=L_{Y_{\bullet}\otimes_{RHc}X_{\bullet}}(F)(M)_{k}.
\end{align*}
If we denote by $\Delta$ the differential of $L_{X_{\bullet}} \circ L_{Y_{\bullet}}(F)$, by $\partial$ the differential of $L_{Y_{\bullet}\otimes_{RHc}X_{\bullet}}(F)$ and by $\delta$ the differential of $L_{Y_{\bullet}}(F)$, we have:
\begin{align*}
\Delta_{k}(M) &= \bigoplus_{n\in \mathbb{Z}} \delta_{n}(X_{n-k}\otimes_{RGb}M) + (-1)^{n}L_{Y}(F)_{n}(d_{n-k+1}^{X}\otimes_{RGb}Id_{M})\\
&=\bigoplus_{n\in\mathbb{Z}}\bigg(\bigoplus_{m\in\mathbb{Z}} \eta_{m}(Y_{m-n}\otimes_{RHc} X_{n-k}\otimes_{RGb} M) \bigg)\\
&+ (-1)^{m} F_{m}(d_{m-n+1}^{Y}\otimes_{RHc} Id_{X_{n-k}} \otimes_{RGb} Id_{M})\\
&+ (-1)^{m} (-1)^{n-m} F_{m}(Id_{Y_{m-n}}\otimes_{RHc} d^{X}_{n-k+1}\otimes_{RGb}Id_{M}  \bigg)\\
&=\bigoplus_{m\in\mathbb{Z}}\bigg( \eta_{m}\big((Y_\bullet \otimes_{RHc} X_{\bullet})_{m-k}\otimes_{RGb} M\big)\\
&+ (-1)^{m} F_{m}\big(\bigoplus_{m\in\mathbb{Z}} d^{Y}_{m-n}\otimes_{RHc}Id_{X_{m-k+1}}\otimes_{RGb}Id_{M}\big)\\
&=(-1)^{m} F_{m}\big(\bigoplus_{m\in\mathbb{Z}} (-1)^{m-n} Id_{Y_{m-n}}\otimes_{RHc}d^{X}_{n-k+1}\otimes_{RGb} Id_{M}\big)\bigg)\\
&=\bigoplus_{m\in\mathbb{Z}}\bigg( \eta_{m}\big((Y_\bullet \otimes_{RHc} X_{\bullet})_{m-k}\otimes_{RGb} M\big) + (-1)^{m}F_{m}\big(D_{m-k+1}\otimes_{RGb}Id_{M}\big)\bigg)\\
&=\partial_{k}(M). 
\end{align*}
Since the isomorphism $L_{X}\circ L_{Y}(F) \cong L_{Y\otimes_{RHc} X}(F)$ basically involves only some isomorphisms of the form $F(V\oplus W)\cong F(V)\oplus F(W)$, for some $RKd$-modules, which are functorial in $F$, the isomorphism $L_{X}\circ L_{Y}(F) \cong L_{Y\otimes_{RHc} X}(F)$ is functorial in $F$. 
\end{proof}
\begin{de}
Let $G$ and $H$ be two finite groups. Let $b$ be a block of $RG$ and $c$ be a block of $RH$. Then a \emph{permeable} derived equivalence between $RGb$ and $RHc$ is:
\begin{enumerate}
\item A bounded complex $X$ of $RGb$-$RHc$-bimodules, which are projective as $RGb$-module and as $RHc$-module, such that:
\begin{itemize}
\item $X\otimes_{RHc} X^{*} \cong RGb$ in the homotopy category of $RGb$-bimodules. That is there exist a contractile complex $C$ of (permeable) $RGb$-bimodules such that
\begin{equation*}
X\otimes_{RHc} X^{*} = RGb \oplus C.
\end{equation*} 
\item $X^{*}\otimes_{RGb} X\cong RHc$ in the homotopy category of $RHc$-bimodules. That is there exist a contractile complex $C'$ of (permeable) $RHc$-bimodule such that
\begin{equation*}
X^{*}\otimes_{RGb} X = RHc \oplus C'.
\end{equation*} 
\end{itemize}
\item All the terms of the complexes $X$ and $X^{*}$ are permeable bimodules. 
\end{enumerate}
The complexes $X$ and $X^{*}$ are called permeable (two-sided) tilting complexes. 
\end{de}
\begin{re}
It is clear that a splendid derived equivalence (see \cite{splendid}) is a permeable equivalence since all the terms of the tilting complex are $p$-permutation bimodules.
\end{re}
\begin{lemma}
Let $X_{\bullet}$ be a bounded complex which induces a permeable derived equivalence between $RHc$ and $RGb$. Then the functor $L_{X_{\bullet}}$ induces a functor from $K^{b}(proj(Fun_{R}^{+}(c)))$ to $K^{b}(proj(Fun_{R}^{+}(b)))$. 
\end{lemma}
\begin{proof}
The finitely generated projective objects of the category $Fun_{R}^{+}(c)$ are the Yoneda functors, that is $Y_{V}=Hom_{RHc}(-,V),$ where $V$ is a finitely generated $p$-permutation $RHc$-module. Let $(F_{\bullet},\eta_{\bullet})$ be a right bounded complex of Yoneda functors. That is the non-zero terms are of the form $F_{i} = Hom_{RHc}(-,V_{i})$ for a finitely generated $p$-permutation $RHc$-module $V_{i}$. Let $M$ be a $p$-permutation $RGb$-module. Then, we have:
\begin{align*}
\big(L_{X_{\bullet}}(F_{\bullet})\big)_{k}(M) &= \bigoplus_{i-j=k} F_{i}(X_{j}\otimes M) \\
&=\bigoplus_{i-j=k} Hom_{RH}(X_{j}\otimes_{RG} M, V_{i})\\
&\cong \bigoplus_{i-j=k} Hom_{RG}(M,Hom_{RH}(X_{j},V_{i})). 
\end{align*}
Since $X_{j}$ is projective as $RHc$-module, we have, by Corollary $9.4.2$ \cite{derived_zimmermann}, an isomorphism of functors
 $$Hom_{RHc}(X_{j},-) \cong Hom_{R}(X_{j},R)\otimes_{RHc} - .$$
 Now, $Hom_{R}(X_{j},R)$ is a permeable bimodule. Then the $RGb$-module 
 $$Hom_{RH}(X_{j},V_{i})\cong Hom_{R}(X_{j},R)\otimes_{RHc} V_{j}$$ 
 is a $p$-permutation $RGb$-module.\newline Since the isomorphism $Hom_{RH}(X_{j}\otimes_{RG} M, V_{i}) \cong Hom_{RG}(M,Hom_{RH}(X_{j},V_{i}))$ is natural in $M$, we have an isomorphism of functors $$\big(L_{X_{\bullet}}(F_{\bullet})\big)_{k} \cong \bigoplus_{i-j=k} Hom_{RG}(-,Hom_{RH}(X_{j},V_{i})),$$
So $L_{X_\bullet}(F_{\bullet})_{k}$ is a (finite) direct sum of finitely generated projective functors.
\end{proof}
\begin{theo}\label{derived_equivalences}
Let $G$ and $H$ be two finites groups, let $b$ be a block of $RG$ and $c$ be a block of $RH$. If the block algebras $RGb$ and $RHc$ are permeable derived equivalent, then the categories $Comack_{R}(b)$ and $Comack_{R}(c)$ are derived equivalent.
\end{theo}
\begin{proof}
It is enough to check that $Fun_{R}^{+}(b)$ and $Fun_{R}^{+}(c)$ are derived equivalent. 
Let $X$ be a permeable tilting complex for $RHc$ and $RGb$. There exist a contractile complex of permeable $RHc$-bimodules such that:
\begin{equation*}
X\otimes_{RGb}X^{*} = RHc \oplus C,
\end{equation*}
Then as functors between the homotopy category $K^{-}(proj(Fun_{R}^{+}(c)))$, we have:
\begin{align*}
L_{X^{*}}\circ L_{X}&\cong L_{X\otimes_{RGb} X^{*}}\\
&\cong L_{RHc\oplus C}\\
&\cong L_{RHc}\oplus L_{C}\\
&\cong L_{RHc}.
\end{align*}
Now, it is clear that $L_{RHc}$ is the identity of $K^{-}(proj(Fun_{R}^{+}(c)))$. Conversely, we have:
\begin{equation*}
L_{X}\circ L_{X^{*}} \cong L_{RGb}.
\end{equation*}
So the homotopy categories $K^{-}(proj(Fun_{R}^{+}(c)))$ and $K^{-}(proj(Fun_{R}^{+}(b)))$ are equivalent (as triangulated categories). By Theorem $6.4$ \cite{rickard_morita}, the categories $D^{b}(Comack_{R}(c))$ and $D^{b}(Comack_{R}(b))$ are equivalent. 
\end{proof}
\section{Applications.}
\subsection{Nilpotent blocks.}
Although the determinant of the Cartan Matrix of a block $b$ of $kG$ is a power of $p$, for the corresponding blocks of the Mackey algebra, it is much more complicated (see \cite{bouc_cartan}). By the results of \cite{tw} this determinant is non zero. However the determinant of the Cartan matrix of a block of a cohomological Mackey algebra can be zero. Bouc in \cite{bouc_cartan} proved that the Cartan matrix of $co\mu_{k}(b)$ is non singular if and only if the block $b$ is a nilpotent block with cyclic defect group. This proof is based on a combinatorial approach, and it may be surprising that nilpotent blocks \emph{and} cyclic defect groups appear in that situation. We will apply Theorem \ref{derived_equivalences} to this situation, and show that it is in fact very natural. 
\newline Let $B$ be a block of $kG$, for an arbitrary finite group $G$. If $B$ is a nilpotent block with defect group $P$, then by Puig's Theorem (see \cite{nilpotent_puig} or \cite{nilpotent_revisited}), there is an isomorphism of $k$-algebras,
\begin{equation*}
B\cong Mat(m,kP),
\end{equation*}
for some $m\in\mathbb{N}$. For the cohomological Mackey algebras, we can lift an equivalence between blocks of group algebras, but for this we need that the equivalence sends $p$-permutation modules to $p$-permutation modules. Unfortunately it is \emph{not} always the case. If the reader is not convinced by this fact he might look at Section $6.2$ of this paper, or at Section $7.4$ of \cite{splendid}.
\newline By the results of sections 7.3 and 7.4 of \cite{splendid} and results of \cite{bouc_dade} and \cite{mazza_thesis}, if $p>2$, or $P$ is abelian (N.B. in fact one can ask weaker condition in case of $p=2$), we can replace the bimodule which gives the Morita equivalence between $B$ and $kP$ by a splendid tilting complex of $B$-$kP$-bimodules.
\begin{coro}
Let $B=kGb$ be a nilpotent block with defect $p$-group $P$. If $p=2$ assume that $P$ is abelian. Then
\begin{equation*}
D^{b}(co\mu_{k}(G)\iota(b)\hbox{-Mod})\cong D^{b}(co\mu_{k}(P)\hbox{-Mod}) \hbox{  as triangulated categories.}
\end{equation*}
\end{coro}
Since the determinant of Cartan matrices is invariant under derived equivalences, the determinant of the Cartan matrix $co\mu_{k}(G)\iota(b)$ is non zero if and only if the determinant of the Cartan matrix $co\mu_{k}(P)$ is non zero. However it is well known that this is the case if and only if the group $P$ is cyclic: indeed the projective indecomposable cohomological Mackey functors for a $p$-group $P$ are $FP_{Ind_{Q}^{P}(k)}$ for $Q\leqslant P$. By adjunction, the coefficient of the Cartan matrix indexed by two projective $FP_{Ind_{Q}^{P}(k)}$ and $FP_{Ind_{Q'}^{P}(k)}$ is:
\begin{align*}
C_{Q,Q'}&=dim_{k}Hom_{kP}(Ind_{Q}^{P}(k),Ind_{Q'}^{P}(k))\\
&=dim_{k}Hom_{kP}(k,Res^{P}_{Q}Ind_{Q'}^{P}k)\\
&= Card([Q\backslash P / Q']).
\end{align*}
By the main result of \cite{thevenaz_cyclic}, this matrix is non degenerate if and only if $P$ is cyclic. 
\subsection{Application to representation's theory of finite groups.}
As immediate, but useful corollary of Proposition \ref{prop_morita}, we have:
\begin{coro}
Let $G$ and $H$ be two finite groups. Let $b$ be a block of $RG$ and $c$ be a block of $RH$. If the cohomological Mackey algebras $co\mu_{R}(G)\iota(b)$ and $co\mu_{R}(H)\iota(c)$ do not have the same Cartan matrix, then $RGb$ and $RHc$ are \emph{not} `splendidly' Morita equivalent. 
\end{coro}
This is useful since there are algorithm which compute these Cartan matrices. By testing this algorithm, the author found an astonishing (at least for him) example of nilpotent blocks with quaternion defect group, where the comportement of the simple modules is rather sophisticated. 
\newline Let $k$ be an algebraically closed field of characteristic $2$. Let $p$ be an odd prime. Let $X_{p^3}$ be an extra-special group of exponent $p$, that is: 
\begin{equation*}
X_{p^3} = {<}a,b,z\ \ ;\ a^p=b^p=z^p=1,\ [a,b]=z,\ [a,z]=[b,z]=1\ {>}.
\end{equation*}
Let $Q_{8}$ be a quaternion group of order $8$, that is:
\begin{equation*}
Q_{8}:={<}i,j\ ;\ i^4=1,\ i^2=j^2,\ j i j^{-1}=i^{-1}{>}
\end{equation*}
Then, one can represent $Q_{8}$ as a subgroup of $GL_{2}(\mathbb{F}_{p})$ by sending $i$ to the matrix $\left(\begin{array}{cc}0 & -1 \\1 & 0\end{array}\right)$ and $j$ to the matrix $\left(\begin{array}{cc}x & y \\y & -x\end{array}\right)$, where $x^2+y^2=-1$. 
\newline A matrix $\left(\begin{array}{cc}\alpha & \beta \\\gamma & \delta\end{array}\right)$ induces an automorphism of $X_{p^3}$ defined by: 
\begin{itemize}
\item $a\mapsto a^{\alpha}b^{\beta}$, 
\item $b\mapsto a^{\gamma}b^{\delta}$,
\item $z\mapsto z^{\alpha\delta-\beta\gamma}. $
\end{itemize}
Let us consider $G=X_{p^3}\rtimes Q_{8}$, where $Q_{8}$ acts on $X_{p^3}$ via its representation in $GL_{2}(\mathbb{F}_{p})$. 
\begin{lemma}
There are $\frac{p^2-1}{8}+p$ blocks of $kG$.
\begin{itemize}
\item $\frac{p^2-1}{8}$ blocks with defect $0$.
\item $p$ nilpotent blocks with $Q_{8}$ as defect group.
\end{itemize}
\end{lemma}
\begin{proof}[Sketch of proof]
Since $X_{p^3}$ is a $2'$-group, the blocks of this group are in bijection with the isomorphism classes of simple modules. There are $p^2-1$ representations which factorise through $C_{p}\times C_{p} = X_{p^3}/D(X_{p^3})$. By usual clifford theory there are $\frac{p^2-1}{8}$ blocks of $kG$ covering all these blocks. Now there are $p-1$ blocks of $kX_{p^3}$ corresponding to the simple modules of dimension $p$, induced by a character of $kC_{p}$. Let $\zeta$ be a $p$-root of $1$ in $k$, then the simple module $V_{\zeta}$ of dimension $p$ is:
\begin{equation*}
V_{\zeta} = Ind_{{<}a,z{>}}^{X_{p^3}}Inf_{{<}z{>}}^{{<}a,z{>}}k_{\zeta}.
\end{equation*}
The inertie group of $V_{\zeta}$ is $G$, so this gives a simple module of $kG$ with $Q_{8}$ as vertex. We denote by $L_{\zeta}$ the $kG$-module such that $Res^{G}_{X_{p^3}}L_{\zeta}\cong V_{\zeta}$. 
\end{proof}
\begin{prop}
Let $\zeta$ be a $p$-root of $1$ in $k$. Let $L_{\zeta}$ be the corresponding simple $kG$-module and let $b_{\zeta}$ be the corresponding block. Then
\begin{itemize}
\item If $p\neq 1\mod 8$, then $kGb_{\zeta}$ is not splendidly Morita equivalent to $kQ_{8}$. 
\item Let $p= 1 \mod 8$ and let $t(\zeta)=\sum_{x\in I} \zeta^{x}$, where $I$ is the set of quadratic residues mod $p$. If $t(\zeta)=0$, then $kGb_{\zeta}$ is splendidly Morita equivalent to $KQ_{8}$.
\item If $t(\zeta)=1$ and $p=17$, then $kGb_{\zeta}$ is \emph{not} splendidly Morita equivalent to $kQ_{8}$.
\end{itemize}
\end{prop}
\begin{re}
The condition $p=17$ appears only because we are not able to find a general proof of this result. However it seems that the result should be true for all $p=1\mod 8$. In particular we check it with GAP in several cases. 
\end{re}
\begin{proof}
Here, we are very sketchy. The first part follows from Mazza's work. See Section $4.2$ of \cite{mazza_endo}. It is showed that $Res^{G}_{Q_{8}}L_{\zeta}$ is an endo-trivial module with source $S_{\zeta}$ such that $dim_{k} S_{\zeta} = p \mod 8$. So if $p\neq 1\mod 8$, then $L_{\zeta}$ is not a $2$-permutation module and the Morita equivalence is not splendid. 
\newline If $p=1\mod 8$, then the source can be either the trivial module or an endo-trivial module of dimension $9$. Let $w\in \mathbb{F}_{p}$ such that $w^2 =1$ and let $b$ be a generators of $\mathbb{F}_{p}^{\times}$. The module $Res^{G}_{Q_{8}}L_{\zeta}$ is a trivial source module if and only if $k$ is a direct summand of $Res^{G}_{Q_{8}} L_\zeta$. This appends if and only if there is a vector $v\in L_{\zeta}^{Q_{8}}$ and an invariant linear form $\phi$ on $Res^{G}_{Q_{8}}L_{\zeta}$ such that $\phi(v)=1$.
\newline Let $i \in \mathbb{F}_{p}^{\times}/{<}w{>}$. Let $t_{i}:= \zeta^{b^i}+\zeta^{wb^{i}}+\zeta^{w^2b^{i}}+\zeta^{w^{3}b^{i}}\in \mathbb{F}_{2^{\frac{p-1}{8}}}$ be a Gaussian sum. Let $M$ be the matrix indexed by $\mathbb{F}_{p}^{\times}/{<}w{>}$, where the $(i,j)$th. coefficient is $t_{i+j}$. One can check that $L_{\zeta}$ is a $2$-permutation module if and only if the constant vector $(1,1,\cdots,1)^{t}$ is in the image of $M-Id$.
\newline Now, if $t(\zeta)=0$, we have $(M-Id)\cdot (1,0,1,0,\cdots 1,0)^{t} = (1,1,\cdots, 1)^{t}$.
\newline If $p=17$ and $t(\zeta)=1$ an easy computation shows that $(1,1,\cdots, 1)$ canot be in the image of $M-Id$. 
\end{proof}
\paragraph{Acknowledgements}
 The author would like to thank the foundation FEDER and the CNRS for their financial support, the foundations ECOS and CONACYT for the financial support in the project M10M01. Thanks also go to Serge Bouc my advisor for numerous helpful conversations. 

\par\noindent
{Baptiste Rognerud\\
EPFL / SB / MATHGEOM / CTG\\
Station 8\\
CH-1015 Lausanne\\
Switzerland\\
e-mail: baptiste.rognerud@epfl.ch}
\end{document}